\documentclass[10pt,reqno]{amsart}

\usepackage{amsmath}
\usepackage{amsthm}
\usepackage{amssymb}
\usepackage{amsfonts}
\usepackage{amsxtra}
\usepackage{fullpage}
\usepackage{amscd}
\usepackage{color} 
\usepackage{bm} 
\usepackage{mathrsfs}
\usepackage{subfigure}
\usepackage{tikz}
\usepackage{tikz-cd}
\usetikzlibrary{cd}
\usepackage{enumerate}
\usepackage[all]{xy}
\usepackage[mathscr]{eucal}
\usepackage{verbatim}
 \usetikzlibrary{shapes.geometric, calc}
 \usepackage{enumitem}
\usetikzlibrary{arrows,decorations.pathmorphing,decorations.pathreplacing,positioning,shapes.geometric,shapes.misc,decorations.markings,decorations.fractals,calc,patterns}

\usepackage{mathtools} 
\usepackage{microtype} 
\usepackage{hyperref}
 \hypersetup{colorlinks,linkcolor=[rgb]{0,0,1},citecolor=[rgb]{1,0,0}}

\usetikzlibrary{chains}

 \tikzset{>=stealth',
        cvertex/.style={circle,draw=black,inner sep=1pt,outer sep=3pt},
        vertex/.style={circle,fill=black,inner sep=1pt,outer sep=3pt},
        star/.style={circle,fill=yellow,inner sep=0.75pt,outer sep=0.75pt},
        tvertex/.style={inner sep=1pt,font=\scriptsize},
        gap/.style={inner sep=0.5pt,fill=white}}
        
\tikzset{node distance=2em, ch/.style={circle,draw,on chain,inner sep=2pt},chj/.style={ch,join},every path/.style={shorten >=4pt,shorten <=4pt},line width=1pt,baseline=-1ex}

\newcommand{\dlabel}[1]{%
	\({{\scriptstyle{ #1}}}\)
}
\newcommand{\dnode}[2][chj]{%
	\node[#1,label={below:\dlabel{ #2}}] {};
}
\newcommand{\dydots}{%
	\node[chj,draw=none,inner sep=1pt] {\dots};
}
\newcommand{\updots}{%
	\node[chj,draw=none,inner sep=1pt] {\vdots};
}
\newcommand{\dnodebr}[1]{%
	\node[chj,label={below right:\dlabel{#1}}] {};
}
\newcommand{\dnodeu}[1]{%
	\node[chj,label={above:\dlabel{#1}}] {};
}

\newtheorem{thm}{Theorem}
\newtheorem{prop}[thm]{Proposition}
\newtheorem{cor}[thm]{Corollary}
\newtheorem{lem}[thm]{Lemma}
\theoremstyle{definition}

\newtheorem{example}[thm]{Example}

\newtheorem{remark}[thm]{Remark}
\newtheorem{remarks}[thm]{Remarks}
\numberwithin{thm}{section}

 \numberwithin{equation}{section} 

\newcommand{\CC}{\mathbb{C}}
\newcommand{\ZZ}{\mathbb{Z}}

\DeclareMathOperator{\End}{\mathrm{End}}
\newcommand{\git}{\ensuremath{/\!\!/\!}}
\DeclareMathOperator{\GL}{\mathrm{GL}}
\DeclareMathOperator{\head}{\mathrm{h}}
\DeclareMathOperator{\Hilb}{{\mathrm{Hilb}}}
\DeclareMathOperator{\Hom}{\mathrm{Hom}}
\DeclareMathOperator{\im}{\mathrm{Im}}
\DeclareMathOperator{\rank}{\mathrm{rank}}
\DeclareMathOperator{\module}{\mathrm{-mod}}
\DeclareMathOperator{\Proj}{\mathrm{Proj}}
\DeclareMathOperator{\Rep}{\mathrm{Rep}}
\DeclareMathOperator{\SL}{\mathrm{SL}}
\DeclareMathOperator{\Sp}{\mathrm{Sp}}
\DeclareMathOperator{\Spec}{\mathrm{Spec}}
\DeclareMathOperator{\Sym}{\mathrm{Sym}}
\DeclareMathOperator{\tail}{\mathrm{t}}

\newcommand{\one}{\ensuremath{(\mathrm{i})}}
\newcommand{\two}{\ensuremath{(\mathrm{ii})}}
\newcommand{\three}{\ensuremath{(\mathrm{iii})}}

\title{Punctual Hilbert schemes for Kleinian singularities \\ as quiver varieties}

\author{Alastair Craw}
\address{Department of Mathematical Sciences, 
University of Bath, 
Claverton Down, 
Bath BA2 7AY, 
United Kingdom}
\email{a.craw@bath.ac.uk}

\author{S{\o}ren Gammelgaard} 
\address{Mathematical Institute, 
University of Oxford, 
Oxford OX2 6GG, 
United Kingdom}
\email{gammelgaard@maths.ox.ac.uk /  szendroi@maths.ox.ac.uk}

\author{\'Ad\'am Gyenge}
\address{Alfr\'ed R\'enyi Institute of Mathematics, Re\'altanoda utca 13-15, 1053, Budapest, Hungary}
\email{Gyenge.Adam@renyi.hu}

\author{Bal\'{a}zs Szendr\H{o}i}


\begin{document}
\begin{abstract}
 For a finite subgroup $\Gamma\subset \SL(2,\CC)$ and $n\geq 1$, we construct the (reduced scheme underlying the) Hilbert scheme of~$n$  points on the Kleinian singularity $\CC^2\!/\Gamma$ as a Nakajima quiver variety for the framed McKay quiver of $\Gamma$, taken at a specific non-generic stability parameter. We deduce that 
 this Hilbert scheme is irreducible (a result previously due to Zheng), normal, and admits a unique symplectic resolution. More generally, we introduce a class of algebras obtained from the preprojective algebra of the framed McKay quiver by removing an arrow and then `cornering',
 and we show that fine moduli spaces of cyclic modules over these new algebras are isomorphic to quiver varieties for the framed McKay quiver and certain non-generic choices of stability parameter.  
\end{abstract}
\maketitle
\tableofcontents
\section{Introduction}
Let $\Gamma\subset \SL(2,\CC)$ be a finite subgroup. One can associate various Hilbert schemes to the action of $\Gamma$ on the affine plane $\CC^2$ and the Kleinian singularity $\CC^2/\Gamma$. 
For $N\coloneqq \vert \Gamma\vert$ and any natural number $n$, the action of $\Gamma$ on $\CC^2$ induces an action of $\Gamma$ on the Hilbert scheme $\Hilb^{[nN]}(\CC^2)$ of $nN$ points on the affine plane. The scheme $n\Gamma\text{-}\Hilb(\CC^2)$, parametrising $\Gamma$-invariant ideals $I$ in
$\CC[x,y]$
such that the quotient $\CC[x,y]/I$ is isomorphic to the direct sum of $n$ copies of the regular representation of $\Gamma$, is a union of components of the fixed point set of the $\Gamma$-action on $\Hilb^{[nN]}(\CC^2)$. It is thus nonsingular and quasi-projective. 
One may also consider 
the Hilbert scheme of $n$ points $\Hilb^{[n]}(\CC^2/\Gamma)$ on the singular surface $\CC^2/\Gamma$, parametrising ideals 
in the invariant ring $\CC[x,y]^\Gamma$ that have codimension $n$.
This Hilbert scheme is quasi-projective, 
 and in this introduction we endow it with the {\em reduced} scheme structure.

These two kinds of Hilbert schemes are related by the morphism
\begin{equation}\label{eg:equivariant_HC}
 n\Gamma\text{-}\Hilb(\CC^2)\longrightarrow \Hilb^{[n]}(\mathbb{C}^2/\Gamma)
 \end{equation} 
 sending a $\Gamma$-invariant ideal $I$ in $\CC[x,y]$ to the ideal $I\cap \CC[x,y]^\Gamma$; this set-theoretic map is indeed a morphism of schemes by Brion~\cite[3.4]{Brion13}. By composing with the Hilbert-Chow morphism of the surface $\CC^2/\Gamma$, we see that~\eqref{eg:equivariant_HC} is in fact a morphism of schemes over the affine scheme $\Sym^n(\CC^2/\Gamma)$.

Until recently, not much was known about the schemes $\Hilb^{[n]}(\mathbb{C}^2/\Gamma)$ for $n>1$. Gyenge, N\'emethi and Szendr\H oi~\cite{GNS} computed the generating function of their Euler characteristics for $\Gamma$ of type A and D (the cyclic and dihedral cases), giving an answer with modular properties. They also conjectured an analogous formula for type E. Zheng~\cite{Zheng17} proved that $\Hilb^{[n]}(\CC^2/\Gamma)$ is always irreducible, and gave a homological characterisation of its smooth points through a detailed 
analysis of Cohen-Macaulay modules over $\CC^2/\Gamma$. Yamagishi~\cite{Yamagishi17} studied symplectic resolutions of the Hilbert squares $\Hilb^{[2]}(\CC^2/\Gamma)$, and described completely the central fibres of these resolutions, from which he deduced that $\Hilb^{[2]}(\CC^2/\Gamma)$ admits a unique symplectic resolution.

 The aim of our paper is to study the spaces appearing in~\eqref{eg:equivariant_HC}, and
all possible ways in which the morphism from \eqref{eg:equivariant_HC} can be decomposed, 
using quiver-theoretic techniques in a uniform way. 
The starting point is the McKay correspondence, which associates a quiver (oriented graph) to the subgroup $\Gamma\subset \SL(2,\CC)$. Representation spaces of a framed variant of the McKay quiver, each depending on a stability parameter, were introduced in Kronheimer and Nakajima~\cite{KN90} and studied further by Nakajima~\cite{Nakajima94}.
Subsequently, for any $n\geq 1$ and for a special choice of framing, Kuznetsov~\cite{Kuznetsov07} determined a pair of cones $C_\pm$ in the space of stability parameters for which the corresponding representation space $\mathfrak{M}_\theta$ is isomorphic to the punctual Hilbert scheme $\Hilb^{[n]}(S)$ of the minimal resolution $S$ of $\CC^2/\Gamma$ for $\theta\in C_-$, and to the scheme $n\Gamma\text{-}\Hilb(\CC^2)$ from \eqref{eg:equivariant_HC} for $\theta\in C_+$, respectively. Much more recently, Bellamy and Craw~\cite{BC18} gave a complete description of the wall-and-chamber structure on the space of stability parameters in this situation, and identified a simplicial cone $F$ containing $C_\pm$ that is isomorphic as a fan to the movable cone of $n\Gamma\text{-}\Hilb(\CC^2)$ for $n>1$; in particular, chambers in this simplicial cone correspond one-to-one with projective, symplectic resolutions of $\Sym^n(\CC^2/\Gamma)$ (see Figure~\ref{fig:A2n=3} below for an example). 

The main result of our paper reconstructs the morphism from \eqref{eg:equivariant_HC} by variation of GIT quotient. Explicitly, we vary a generic stability parameter $\theta\in C_+$ to a parameter $\theta_0$ in a particular extremal ray of the closure of $C_+$; the induced morphism $\mathfrak{M}_\theta\to \mathfrak{M}_{\theta_0}$ coincides with the morphism \eqref{eg:equivariant_HC}. As a corollary, we obtain the following result.

\begin{thm}
\label{thm:mainintro}
Let $\Gamma\subset \SL(2,\mathbb{C})$ be a finite subgroup and let $n\geq 1$. The (reduced) Hilbert scheme $\Hilb^{[n]}(\CC^2/\Gamma)_{\rm red}$ is an irreducible, normal scheme 
with symplectic, hence rational Gorenstein, singularities. 
Furthermore, it admits a unique projective, symplectic resolution given by \eqref{eg:equivariant_HC}.
\end{thm}

We reiterate that irreducibility is due originally to Zheng~\cite{Zheng17}. The existence of a nowhere-vanishing $2n$-form in the type $A$ case, which follows from having symplectic singularities, was shown in the same paper \cite[Theorem~D]{Zheng17}, while the existence and uniqueness of the symplectic resolution for $n=2$ is due to Yamagishi~\cite{Yamagishi17}.

 Our main tool is to furnish $\Hilb^{[n]}(\CC^2/\Gamma)$ with a quiver-theoretic interpretation as a fine quiver moduli space by the process of cornering~\cite{CIK18}.
 More generally, we provide a fine moduli space description of the quiver varieties $\mathcal{M}_\theta$ for all non-generic stability parameters that lie in the closure of the cone $C_+$. Our methods give conceptual proofs of the geometric properties of $\Hilb^{[n]}(\CC^2/\Gamma)$ listed in Theorem~\ref{thm:mainintro}, and allow us to obtain all possible projective factorisations of the morphism \eqref{eg:equivariant_HC} by universal properties of the resulting fine moduli spaces. Our proofs for all statements concerning $\Hilb^{[n]}(\CC^2/\Gamma)$ avoid case-by-case analysis with respect to the ADE classification of $\Gamma$. We use only one such case-by-case argument for our more general quiver varieties $\mathcal{M}_\theta$ to establish a bound on the dimension vector for quiver representations that are stable with respect to a non-generic stability condition; see Appendix~\ref{sec:the ugly proof}, sections~\ref{sec:remainingCases}--\ref{sec:generalcase}.

Quiver varieties with degenerate stability conditions identical to ours were considered before in~\cite{Nak_branching}. In very recent work, Nakajima~\cite{Nakajima20} uses the main result of our paper and some results from the representation theory of quantum affine algebras to prove the conjecture of~\cite{GNS}.

\smallskip

\noindent {\bf Acknowledgements.} While working on this project, S.G. was supported by an Aker Scholarship, whereas \'A.Gy. and B.Sz. were supported by 
EPSRC grant EP/R045038/1. We thank Gwyn Bellamy, Ben Davison and Hiraku Nakajima for helpful discussions. We are also grateful to the anonymous referee whose helpful comments enabled us to bypass a case-by-case analysis of ADE diagrams in the situation of primary interest where our quiver variety admits a morphism to $\Hilb^{[n]}(\CC^2/\Gamma)$, and who identified a gap in an earlier, more complicated construction of our fine moduli spaces, leading to a substantial simplification.

\smallskip

\noindent {\bf Notation. } Let $\pi\colon X\to Y$ be a projective morphism of schemes over an affine base $Y$. For a globally generated line bundle $L$ on $X$, write $\vert L\vert \coloneqq \Proj_Y\bigoplus_{k\geq 0} H^0(X,L^k)$
for the (relative) linear series of $L$, and 
$\varphi_{\vert L\vert}\colon X\to \vert L\vert$ for the induced morphism over $Y$.

\section{Variation of GIT quotient for quiver varieties}
 Let $\Gamma\subset \SL(2,\CC)$ be a finite subgroup. Let $V$ denote its given two-dimensional representation, defined by this inclusion. Write $\rho_0,\rho_1,\dots,\rho_r$ for the irreducible representations of $\Gamma$, with $\rho_0$ the trivial one. The \emph{McKay graph} of $\Gamma$ has vertex set $\{0, 1, \dots, r\}$ where vertex $i$ corresponds to the representation $\rho_i$ of $\Gamma$, and there are $\dim\Hom_\Gamma(\rho_j,\rho_i\otimes V)$ edges between vertices $i$ and $j$. By the McKay correspondence \cite{McKay80}, the McKay graph is an extended Dynkin diagram of $ADE$ type. Add a framing vertex $\infty$, together with an edge between vertices $\infty$ and $0$, and let $Q_1$ be the set of pairs consisting of an edge in this graph and an orientation of the edge. If $a$ is an edge with orientation, we write $ {a}^*$ for the same edge with the opposite orientation. The \emph{framed McKay quiver} $Q$ has vertex set $Q_0 = \{\infty, 0, 1, \dots, r\}$ and arrow set $Q_1$, where for each oriented edge $a\in Q_1$ we write $ \tail(a), \head(a) $ for the tail and head of $a$ respectively.
  
  Let $\CC Q$ denote the path algebra of $Q$. 
Let $\epsilon\colon Q_1\to \{\pm 1\}$ be any map such that $\epsilon(a) \neq \epsilon(a^*)$ for all $a\in Q_1$. The preprojective algebra $\Pi$ is the quotient of $\CC Q$ by the ideal generated by the relation 
   \[
   \sum_{a\in Q_1} \epsilon(a) aa^*.
   \]
    Equivalently, multiplying both sides of this relation by the vertex idempotent at vertex~$i$ shows that~$\Pi$ can be presented as the quotient of $\CC Q$ by the ideal
    \begin{equation}
        \label{eqn:Pirelations}
    \left( \sum_{\head(a) = i} \epsilon(a) aa^* \mid i\in Q_0\right).
     \end{equation}
     The preprojective algebra $\Pi$ does not depend on the choice of the map~$\epsilon$ \cite[Lemma 2.2]{CBH98}. 
 Let $R(\Gamma)$ denote the representation ring of $\Gamma$. Introduce a formal symbol $\rho_\infty$ so that  $\{\rho_i \mid i\in Q_0\}$ provides a $\ZZ$-basis for $\ZZ^{Q_0}\cong\ZZ\oplus R(\Gamma)$ considered as $\ZZ$-modules. 
 
 For a natural number $n \geq 1$ that we fix for the rest of the paper, consider the dimension vector 
 \[
 v\coloneqq(v_i)_{i\in Q_0} \coloneqq\rho_\infty + \sum_{i\geq 0} n\dim(\rho_i)\rho_i 
 \in \ZZ^{Q_0}.
 \]
 The group $G(v)\coloneqq \mathbb{C}^\times \times \prod_{0\leq i\leq r} \GL\big(n\dim(\rho_i),\CC\big)$ acts on the space $\Rep(Q,v)\coloneqq\bigoplus_{a\in Q_1} \Hom\big(\mathbb{C}^{v_{\tail(a)}},\mathbb{C}^{v_{\head(a)}}\big)$ of representations of the quiver $Q$ of dimension vector $v$ by conjugation. The diagonal scalar subgroup acts trivially, and the action of the quotient $G\coloneqq G(v)/\CC^\times$ induces a moment map $\mu\colon \Rep(Q,v)\to \mathfrak{g}^*$ such that a closed point lies in $\mu^{-1}(0)$ if and only if the corresponding $\CC Q$-module satisfies the relations \eqref{eqn:Pirelations} of the preprojective algebra $\Pi$. If we write 
 \[
 \Theta_v = \{\theta\colon \mathbb{Z}^{Q_0}\to \mathbb{Q} \mid \theta(v)=0\},
 \]
 then each character of $G$ is $\chi_\theta\colon G\to \CC^\times$ for some integer-valued $\theta\in \Theta_v$, where $\chi_\theta(g) = \prod_{i\in Q_0} \det(g_i)^{\theta_i}$ for $g\in G(v)$. 
 
Given a stability parameter $\theta\in \Theta_v$, recall that a $\Pi$-module $M$ is $\theta$-stable (respectively semistable) if $\theta(\dim M) = 0$ and 
for every proper, nonzero submodule $N\subset M$, we have $\theta(\dim N) > 0$ (respectively $\theta(\dim N) \ge 0$).  
Two $\theta$-semistable $\Pi$-modules $M, M'$ are said to be $S$-equivalent, if they admit filtrations
 \[0=M_0\subset M_1\subset \dots \subset M_{s_1} = M\qquad \textrm{and}\qquad 0=M_0'\subset M_1'\subset \dots\subset M_{s_2}' = M'\]
 such that each $M_i$ and each $M_j'$ is $\theta$-semistable, and 
 \[\bigoplus_{i=1}^{s_1} M_i/M_{i-1} \cong \bigoplus_{i=1}^{s_2} M_i'/M_{i-1}'.\]
 Every $S$-equivalence class has a representative unique up to isomorphism that is a direct sum of $\theta$-stable modules, the so-called \emph{polystable} module.

 Given $\theta\in \Theta_v$, the quiver variety 
 \[
 \mathfrak{M}_\theta\coloneqq (\mu^{-1}(0)\git_\theta G)_{\rm red}
 \]
 is the categorical quotient of the locus of $\chi_\theta$-semistable points of $\mu^{-1}(0)$ by the action of $G$. It is the coarse moduli space of S-equivalence classes of $\theta$-semistable $\Pi$-modules of dimension vector $v$. As indicated, we consider these GIT quotients with their reduced scheme structure everywhere below.
  
 \begin{lem}
 \label{lem:Mthetageometry}
 For all $\theta\in \Theta_v$, the scheme $\mathfrak{M}_\theta$ is irreducible and normal, with symplectic singularities.
 \end{lem}
 \begin{proof}
See Bellamy and Schedler~\cite[Therorem 1.2, Proposition~3.21]{BS16}.
 \end{proof}
 
 The set of stability conditions $\Theta_v$ admits a preorder $\ge$,
 where $\theta\geq \theta'$ iff every $\theta$-semistable $\Pi$-module is $\theta'$-semistable. It is well known \cite{DH98, Thaddeus96} that we obtain a wall-and-chamber structure on $\Theta_v$, where $\theta, \theta'\in \Theta_v$ lie in the relative interior of the same cone if and only if
 both $\theta\ge\theta^\prime$ and $\theta^\prime\ge \theta$ hold in this preorder,
 in which case $\mathfrak{M}_\theta\cong \mathfrak{M}_{\theta'}$. The interiors of the top-dimensional cones in $\Theta_v$ are \emph{GIT chambers}, while the codimension-one faces of the closure of each GIT chamber are \emph{GIT walls}. We say that $\theta\in \Theta_v$ is \emph{generic} with respect to $v$, if it lies in some GIT chamber; equivalently, $\theta$ is generic if every $\theta$-semistable $\Pi$-module is $\theta$-stable. Since $v$ is indivisible,  King~\cite[Proposition~5.3]{King94} proves that for generic $\theta\in \Theta_v$, the quiver variety $\mathfrak{M}_\theta$ is the fine moduli space of isomorphism classes of $\theta$-stable $\Pi$-modules of dimension vector~$v$. In this case, the universal family on $\mathfrak{M}_\theta$ is a tautological locally-free sheaf
 \[
 \mathcal{R}\coloneqq\bigoplus_{i\in Q_0} \mathcal{R}_i
 \]
 together with a $\CC$-algebra homomorphism $\phi\colon \Pi\to \End(\mathcal{R})$,  where $\mathcal{R}_\infty$ is the trivial bundle on $\mathfrak{M}_\theta$ and where $\rank(\mathcal{R}_i)=n\dim(\rho_i)$ for $i\geq 0$, 
 
  Variation of GIT quotient for the quiver varieties $\mathfrak{M}_\theta$ was investigated recently by the first author with Bellamy~\cite{BC18}. The following result records a surjectivity statement that will be useful later on.
 
  \begin{lem}
  \label{lem:VGIT}
  Let $\theta, \theta'\in \Theta_v$ satisfy $\theta\geq \theta^\prime$. Then the morphism $\pi\colon \mathfrak{M}_\theta\to \mathfrak{M}_{\theta'}$ obtained by variation of GIT quotient is a surjective, projective and birational morphism of varieties over $\Sym^n(\mathbb{C}^{2}/\Gamma)$. 
  \end{lem}
  \begin{proof}
  If $\theta$ is generic and $\theta'=0$, then the morphism $\mathfrak{M}_\theta\to\mathfrak{M}_0\cong \Sym^n(\mathbb{C}^{2}/\Gamma)$ is a projective symplectic resolution \cite[Theorem~4.5]{BC18} and the result holds. For the general case, 
  combining \cite[Lemma~3.22]{BS16} and \cite[Lemma~4.4]{BC18}, we get $\dim \mathfrak{M}_{\theta} = 2n$. This holds for any $\theta\in \Theta_v$, so $\dim\mathfrak{M}_{\theta'}=2n$.  The morphism $\pi\colon \mathfrak{M}_\theta\to \mathfrak{M}_{\theta'}$ is projective, so the image $Z\coloneqq \pi(\mathfrak{M}_\theta)$ is closed in $\mathfrak{M}_{\theta'}$. Deform $\theta$ if necessary to a generic parameter $\eta$ such that $\eta\geq \theta$. Then the resolution $\mathfrak{M}_\eta\to\mathfrak{M}_0\cong \Sym^n(\mathbb{C}^{2}/\Gamma)$ factors through $\pi$ by variation of GIT quotient, so $\dim(Z) = 2n$ and hence $\pi$ is birational onto its image. It follows that $Z$ is an irreducible component of $\mathfrak{M}_{\theta'}$. However, $\mathfrak{M}_{\theta'}$ is irreducible \cite[Proposition~3.21]{BS16}; so $\pi$ is surjective.
  \end{proof}

  The GIT wall-and-chamber structure on $\Theta_v$ was computed explicitly in \cite[Theorem~4.6]{BC18}. In this paper, we focus on the distinguished GIT chamber
   \begin{equation}
       \label{eqn:C_+}
   C_+\coloneqq \big\{\theta\in \Theta_v \mid \theta(\rho_i) > 0 \text{ for }i\geq 0\big\}.
   \end{equation}
  It is well known that the quiver variety $\mathfrak{M}_\theta$ for $\theta\in C_+$ admits a description as an equivariant Hilbert scheme. Recall from the Introduction that $n\Gamma\text{-}\Hilb(\CC^2)$ is the scheme parametrising $\Gamma$-invariant ideals $I\lhd\CC[x,y]$ with quotient isomorphic as a representation of $\Gamma$ to the direct sum of $n$ copies of the regular representation of $\Gamma$.

 \begin{thm}[\cite{VV99, Wang99, Kuznetsov07}]
\label{thm:VVW}
 Let $\Gamma_n\coloneqq\Gamma^n\rtimes \mathfrak{S}_n\subset \Sp(2n,\CC)$ denote the wreath product of $\Gamma$ with the symmetric group $\mathfrak{S}_n$. For $\theta\in C_+$, there is a commutative diagram
  \begin{equation*}
\begin{tikzcd}
 n\Gamma\text{-}\Hilb(\CC^2) \ar[r,"\sim"]\ar[d] & \mathfrak{M}_\theta \ar[d,"{\pi}"] \\
 \CC^{2n}/\Gamma_n\cong \Sym^n(\CC^{2}/\Gamma)\ar[r,"\sim"] &  \mathfrak{M}_0
 \end{tikzcd}
\end{equation*}
 in which the horizontal arrows are isomorphisms and the vertical arrows are symplectic resolutions. 
 \end{thm}
 
 We now study partial resolutions of $\Sym^n(\CC^{2}/\Gamma)$ through which the resolution from Theorem~\ref{thm:VVW} factors. The result of \cite[Proposition~6.1]{BC18} implies that for $n>1$, the nef cone of $n\Gamma\text{-}\Hilb(\CC^2)$ over $\Sym^n(\CC^{2}/\Gamma)$ is isomorphic to the closure $\overline{C_+}$ of the chamber from \eqref{eqn:C_+}. For $n=1$, the relation between these two cones is described in \cite[Proposition~7.11]{BC18} (see Remark~\ref{rem:n=1} for more on the case $n=1$). In any case, for $n\geq 1$, the partial resolutions of interest can all be obtained as follows: choose a face of $\overline{C_+}$ and any GIT parameter from the relative interior of that face; then perform variation of GIT quotient as the parameter moves to the origin in $\Theta_v$. 
 
 Every face of $\overline{C_+}$ is of the form 
   \[
   \sigma_J\coloneqq\left\{\theta\in \overline{C_+} \mid \theta(\rho_j)>0 \text{ iff }j\in J\right\}
   \]
   for some (possibly empty) subset $J\subseteq \{0,1,\dots,r\}$. The parameter $\theta_J\in \overline{C_+}$ defined by setting 
   \[
   \theta_J(\rho_i) = \left\{\begin{array}{cr} -\sum_{j\in J}n\dim(\rho_j) & \text{ for } i=\infty \\ 1 & \text{ if }i\in J \\ 0 & \text{ if }i\in \{0,1,\dots,r\}\setminus J\end{array}\right.\]
   lies in the relative interior of 
   $\sigma_J$. To simplify notation, in the case $J=\{0\}$ we occasionally denote $\theta_0\coloneqq\theta_{\{0\}}$.
   
   \begin{prop}
   \label{prop:poset}
   The face poset of the cone $\overline{C_+}$ can be identified with the poset on the set of quiver varieties $\mathfrak{M}_{\theta_J}$ for subsets $J\subseteq \{0,1,\dots,r\}$, where edges in the Hasse diagram of the poset are realised by the surjective, projective and birational morphisms $\pi_{J,J'}\colon \mathfrak{M}_{\theta_J}\to \mathfrak{M}_{\theta_{J'}}$. 
\end{prop}
   \begin{proof}
   This is standard for variation of GIT quotient apart from surjectivity and birationality of each $\pi_{J,J'}$. This was established in Lemma~\ref{lem:VGIT}. 
   \end{proof}
  
   \begin{remark}
   When $J'=\varnothing$ and $J=\{0,\dots,r\}$, the morphism $\mathfrak{M}_{\theta_J}\to\mathfrak{M}_{\theta_{J'}}$ is the resolution $n\Gamma\text{-Hilb}(\CC^2)\to \Sym^n(\CC^{2}/\Gamma)$ from Theorem~\ref{thm:VVW}. The statement of Proposition~\ref{prop:poset} implies that the paths in the Hasse diagram of the face poset of $\overline{C_+}$ from the unique maximal element to the unique minimal element provide all possible ways in which this resolution can be decomposed via primitive morphisms \cite{Wilson92}. 
 \end{remark}
   
   \begin{example} 
   Consider the case $\Gamma\cong\mu_3$, corresponding to Dynkin type $A_2$, and $n=3$. Figure~\ref{fig:A2n=3} shows a transverse slice of the GIT wall-and-chamber structure inside a specific closed cone $F$ in the space $\Theta_v$ of stability parameters. According to~\cite[Theorem 1.2]{BC18}, this decomposition of the cone is isomorphic as a fan to the closure of the movable cone of this particular $n\Gamma\text{-}\Hilb(\CC^2)$, with its natural subdivision into nef cones of birational models. The open subcone $C_+$ corresponds to the ample cone of $n\Gamma\text{-}\Hilb(\CC^2)$ itself. In Section~\ref{sec:HilbSchemesOfSings} we focus on the distinguished ray $\langle\theta_0\rangle$ in the boundary of $F$. 
\end{example}
  \begin{figure}[!ht]
   \centering
       \begin{tikzpicture}[baseline={(0,0)},xscale=1.25,yscale=1]
			\tikzset{>=latex}
            \draw (4,4)--(2,0)--(0,4)
            --(4,4)
            --(10/11,24/11)
            --(34/11,24/11)
            --(0,4)
            --(2.75,1.5)
            --(1.25,1.5)
            --(4,4)
            --(0,4)
            -- cycle;
              \node at (2,1) {$C_+$};
 \filldraw(2,0) circle (1pt) node[align=left,   below] {$\langle\theta_0\rangle$};
			\end{tikzpicture}
           \caption{Wall-and-chamber structure inside the cone $F$ for $\Gamma\cong\mu_3$ and $n=3$}
            \label{fig:A2n=3}
  \end{figure}
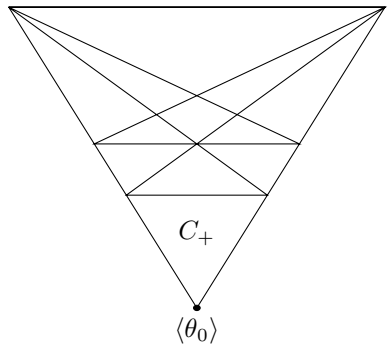
  
  We conclude this section with a lemma that identifies the key geometric fact that makes the chamber $C^+$ special; our argument depends crucially on this observation. For $\theta\in C_+$ and for any $\theta'\in \Theta_v$, we consider the line bundle $L_{C_+}(\theta')\coloneqq \bigotimes_{0\leq i\leq r} \det(\mathcal{R}_i)^{\theta'_i}$ on $\mathfrak{M}_\theta$; the line bundle   $L_J\coloneqq L_{C_+}(\theta_J)$
  will play a special role in particular.
  
  \begin{lem}
  \label{lem:ampleVGIT}
 Let $\theta\in C_+$. Then
 \begin{enumerate}
 \item[\one] for each $\theta' \in \overline{C_+}$, the line bundle $L_{C_+}(\theta')$ on $\mathfrak{M}_\theta$ is globally generated;
 \item[\two] for any $J\subseteq \{0,\dots, r\}$, after multiplying $\theta_J$ by a positive integer if necessary, the morphism to the linear series of $L_J$ decomposes as the composition of $\pi_J$ and a closed immersion:
   \begin{equation}
 \label{eqn:ampleVGIT}
\begin{tikzcd}
  \mathfrak{M}_\theta \ar[d,swap,"{\pi_J}"]\ar[r,"\varphi_{\vert L_J\vert}"] &  \vert L_J\vert. \\   \mathfrak{M}_{\theta_J}\ar[ur,hook] & 
  \end{tikzcd}
  \end{equation}
 \end{enumerate}
  \end{lem}
  \begin{proof}
   Since $\theta\in C_+$, the tautological bundles $\mathcal{R}_i$ on the quiver variety $\mathfrak{M}_\theta$ are globally generated for $i\in {Q_0}$ by \cite[Corollary~2.4]{CIK18}. Hence $L_{C_+}(\theta')$ is globally generated because $\theta_i'\geq 0$ for all $0\leq i\leq r$. In particular, since $\theta_J\in \overline{C_+}$, the rational map $\varphi_{\vert L_J\vert}$ is defined everywhere. 
   The line bundle $L_J$ induces the morphism $\pi_J\colon \mathfrak{M}_\theta\to \mathfrak{M}_{\theta_J}\subset \vert L_J\vert$ by \cite[Theorem~1.2]{BC18}, where we take a positive multiple of $\theta_J$ if necessary to ensure that the polarising ample bundle on $\mathfrak{M}_{\theta_J}$ is very ample. This proves the result.
  \end{proof}
  
  \begin{remark}
  \label{rem:highmultiple}
  We choose a sufficiently high multiple of $\theta$ (and the same high multiple of each $\theta_J$) to ensure that the polarising ample line bundle on $\mathfrak{M}_{\theta_J}$ is very ample for every subset $J\subseteq \{0,\dots, r\}$.
   \end{remark}

  \section{Deleting an arrow and cornering the algebra}

 In the framed McKay quiver $Q$, let $b^*\in Q_1$ be the unique arrow with head at vertex $\infty$. Define a new $\CC$-algebra as the quotient of the preprojective algebra $\Pi$ by the two-sided ideal generated by the class of $b^*$:
 \[
 A\coloneqq \Pi/(b^*).
 \]
 Equivalently, if we define a quiver $Q^*$ to have vertex set $Q^*_0=\{\infty,0,1,\dots,r\}$ and arrow set $Q_1^*=Q_1\setminus \{b^*\}$, then $A$ is the quotient of the path algebra $\CC Q^*$ by the ideal of relations
  \begin{equation}
        \label{eqn:Arelations}
    \left( \sum_{\head(a) = i} \epsilon(a) aa^* \mid 0\leq i, \head(a),\head(a^*)\leq r \right).
     \end{equation}
 Since $Q^*$ and $Q$ share the same vertex set, we may consider $v\in \ZZ^{Q_0}$ as a dimension vector for $A$-modules and any parameter $\theta\in \Theta_v$ as a stability condition for $A$-modules of dimension vector $v$. 
 
 \begin{lem}
 \label{lem:PiA} 
For $\theta\in C_+$, the tautological $\CC$-algebra homomorphism $\Pi\to \End(\mathcal{R})$ over $\mathfrak{M}_\theta$ factors through~$A$. In particular, the tautological bundle $\mathcal{R}$ on $\mathfrak{M}_\theta$ is a flat family of $\theta$-stable $A$-modules of dimension vector~$v$.
 \end{lem}
 \begin{proof}
 The image of the arrow $b^*$ under the tautological homomorphism $\Pi\to \End(\mathcal{R})$ is a map of vector bundles $\mathcal{R}_0\to \mathcal R_\infty\cong \mathcal{O}_{\mathfrak{M}_\theta}$. Under the isomorphism $\mathfrak{M}_\theta \cong n\Gamma\text{-}\Hilb(\CC^2)$ from Theorem~\ref{thm:VVW}, we may regard the fibre of $\mathcal{R}$ over any closed point as the quotient $\CC[x,y]/I$ for some $\Gamma$-invariant ideal $I$ of $\CC[x,y]$. In this language, the restriction to the $0$ vertex is the $\Gamma$-invariant part of the 
quotient $\CC[x,y]/I$, and it is well known that the induced map to the one-dimensional vector space at the vertex $\infty$ vanishes in this case. Thus, as $\mathfrak{M}_\theta$ is non-singular and in particular reduced,
the corresponding map $\mathcal{R}_0\to \mathcal{O}_{\mathfrak{M}_\theta}$ is the zero map, so the tautological $\CC$-algebra homomorphism $\Pi\to \End(\mathcal{R})$ factors through $A$ as required.
 \end{proof}
 
  Lemma~\ref{lem:PiA} allows us to work with the algebra $A$ rather than $\Pi$. To illustrate why this is convenient, define the \emph{unframed McKay quiver} $Q_\Gamma$ to be the complete subquiver of $Q^*$ on the vertex set $\{0,1,\dots,r\}$; that is, if $b\in Q_1^*$ is the unique arrow with tail at $\infty$, then $Q_\Gamma$ has vertex set $\{0,\dots, r\}$ and arrow set $Q_1^*\setminus \{b\}$. The preprojective algebra $\Pi_\Gamma$ of the unframed McKay quiver is the quotient of $\CC Q_\Gamma$ by the ideal generated by the preprojective relations in $Q_\Gamma$ defined similarly to those in $Q$ as in \eqref{eqn:Pirelations}. Now, while $\Pi_\Gamma$ is not a subalgebra of~$\Pi$, it is isomorphic to a subalgebra of~$A$ as we now show. Denote by $e_i\in A$ the vertex idempotents in the algebra~$A$, and by a slight abuse of notation, let $b\in A$ denote the class of the arrow $b$.

   \begin{lem}
   \label{lem:PiJsubalg}
 The preprojective algebra $\Pi_\Gamma$ is isomorphic to the subalgebra $\bigoplus_{0\leq i,j\leq r} e_j A e_i$ of $A$. Under this embedding, there is a isomorphism
\[  A \cong \Pi_\Gamma \oplus \Pi_\Gamma b \oplus \mathbb{C} e_\infty
\]
 of complex vector spaces.
  \end{lem}
   \begin{proof}
   The quiver $Q^*$ has no arrow with head at vertex $\infty$, so the subalgebra $\bigoplus_{0\leq i,j\leq r} e_j A e_i$ of $A$ is isomorphic to the quotient algebra $A/(e_\infty)$. The first statement follows since $A/(e_\infty) \cong \Pi_\Gamma$. The decomposition of $A$ as a vector space is immediate from the structure of the quiver $Q^*$.
   \end{proof}
   
  Let $J\subseteq \{0,1,\dots,r\}$ be non-empty. Define the idempotent $e_J\coloneqq e_\infty + \sum_{j\in J} e_j$ and consider the subalgebra 
 \[
 A_J\coloneqq e_J A e_J
 \]
 of $A$ spanned over $\CC$ by the classes of paths in $Q^*$ whose tail and head both lie in the set $\{\infty\}\cup J$. The process of passing from $A$ to $A_J$ is called \emph{cornering}; see \cite[Remark~3.1]{CIK18}. 
 
We will study moduli spaces of certain finite-dimensional $A_J$-modules, and for this we must introduce a presentation of the algebra $A_J$ in terms of a quiver with relations.

 First, recall that 
 the $\Gamma$-module $R:=\CC[x,y]$ decomposes into isotypical components $R=\bigoplus_{0\leq i\leq r} R_i$, where $R_i$ is the sum of all $\Gamma$-submodules of $R$ that are isomorphic to $\rho_i$. In particular, $R_0=\CC[x,y]^{\Gamma}$ and $R_i$ is a reflexive $R_0$-module for each $0\leq i\leq r$. Since $\Gamma\subset \SL(2,\CC)$, the $\Gamma$-invariant subring $R_0$ is well known to admit a presentation of the form
  \begin{equation}
 \label{eqn:duValEquation}      R_0\cong \CC[z_1,z_2,z_3]/(f),
  \end{equation} 
 leading to the famous description of $\CC^2/\Gamma$ as a hypersurface $(f=0)\subset \CC^3$. On the other hand, combining Auslander~\cite{Auslander86} and Reiten--Van den Bergh~\cite{RVdB89} (compare Buchweitz~\cite{Buchweitz12}) gives an isomorphism
\begin{equation}
 \label{eqn:ARVdB}
 \Pi_\Gamma\cong \End_{R_0}\bigg(\bigoplus_{0\leq i\leq r} R_i\bigg)
 \end{equation}
 of $\CC$-algebras. Note that for $0\leq i, j\leq r$, the space $\Hom_{R_0}(R_i,R_j)$ is finitely generated as an $R_0$-module.  One way to see this is to consider the reflexive sheaves $\widetilde R_i$ on $\CC^2/\Gamma$ determined by the reflexive $R_0$-modules~$R_i$; then $\Hom_{R_0}(R_i,R_j)$ is the space of sections of the coherent sheaf $\mathcal{H}om(\widetilde R_i,\widetilde R_j)$ on $\CC^2/\Gamma$.
 
 \begin{prop}
 \label{prop:quiver}
For any non-empty subset $J\subseteq \{0,1,\dots,r\}$, the algebra $A_J$ can be presented as the quotient of the path algebra of a quiver modulo a two-sided ideal of relations.
 \end{prop}
 \begin{proof}
If we regard $\Pi_\Gamma$ as a subalgebra of $A$ using Lemma~\ref{lem:PiJsubalg}, we see that $e^\prime_J:= \sum_{j\in J} e_j$ is the sum of vertex idempotents in $\Pi_\Gamma$, 
and the isomorphism \eqref{eqn:ARVdB} induces an isomorphism
\begin{equation}
    \label{eqn:EndJ}
e^\prime_J(\Pi_\Gamma)e^\prime_J \cong \End_{R_0}\Big(\bigoplus_{j\in J} R_j\Big).
\end{equation}
 Since $\Pi_\Gamma$ is isomorphic to $\bigoplus_{0\leq i,j\leq r} e_j A e_i$ by Lemma~\ref{lem:PiJsubalg}, it follows that the algebra $e^\prime_J(\Pi_\Gamma)e^\prime_J$ from \eqref{eqn:EndJ} is isomorphic to the subalgebra $e^\prime_J A_J e_J^\prime$ of $A_J$. For each $j\in J$, we choose three $\CC$-algebra generators of $R_0\subseteq \Hom_{R_0}(R_j,R_j)$ corresponding to the generators in the presentation~\eqref{eqn:duValEquation}, and we extend this to a set of $d_j\geq 3$ generators of $\Hom_{R_0}(R_j,R_j)$ as a $\CC$-algebra.
 Finally, for $i, j\in J$ with $i\neq j$, we choose a finite generating set for $\Hom_{R_0}(R_i,R_j)$ as an $R_0$-module comprising $d_{i,j}>0$ generators.
 
 We claim that there exist quivers $Q_J$ and $Q^*_J$ whose path algebras fit into a commutative diagram
  \begin{equation}
 \label{eqn:algebrahomosJ}
\begin{tikzcd}
  \CC Q_J \ar[d,"\alpha_J"]\ar[r] &  \CC Q_J^* \ar[d,"\beta_J"] \\
  \End_{R_0}\Big(\bigoplus_{j\in J} R_j\Big) \ar[r] & A_J
  \end{tikzcd}
  \end{equation}
 of $\CC$-algebra homomorphisms where the vertical maps are surjective and the horizontal maps are injective. Given the claim, Proposition~\ref{lem:PiJsubalg} follows because the required quiver is $Q_J^*$ and the ideal of relations is $\ker(\beta_J)$.

 To prove the claim, we consider two cases. Suppose first that $0\in J$. Define the vertex set of $Q_J$ to be $J$. For the arrow set of $Q_J$, introduce $d_{j}$ loops at each vertex $j\in J$  corresponding to our chosen $\CC$-algebra generators of $\Hom_{R_0}(R_j,R_j)$, including the three distinguished generators of its subalgebra $R_0$.
  Furthermore, for each $i, j\in J$ with $i\neq j$, introduce $d_{i,j}$ arrows from $i$ to $j$ corresponding to our chosen  $R_0$-module generators of $\Hom_{R_0}(R_i,R_j)$. The concatenation of any arrow from $i$ to $j$ with loops at vertex $j$ corresponding to the appropriate elements of $R_0 \subseteq \Hom_{R_0}(R_j,R_j)$ defines paths in $Q_J$ that represent a spanning set for the vector space  $e_jA_Je_i=\Hom_{R_0}(R_i,R_j)$. This
 determines by construction the left-hand epimorphism $\alpha_J$ in 
 \eqref{eqn:algebrahomosJ}. 
 
 Next, define the vertex set of $Q^*_J$ to be $\{\infty\}\cup J$ and define the arrow set by augmenting the arrow set of $Q_J$ with one additional arrow $b$ with tail at $\infty$ and head at $0$. This is well-defined since $0\in J$, and moreover, $\CC Q_J$ is a subalgebra of $\CC Q_J^*$ because $Q_J^*$ has no arrows with head at $\infty$. The lower horizontal map in diagram \eqref{eqn:algebrahomosJ} is simply the inclusion of the algebra from \eqref{eqn:EndJ} as the subalgebra $e^\prime_J A_J e_J^\prime$ of $A_J$. To construct $\beta_J$, we need only extend $\alpha_J$ by sending the paths $e_\infty, b$ in $\CC Q_J^*$ to the classes of $e_\infty,b$ in $A_J$ respectively. Surjectivity of $\beta_J$ follows from the second statement of Lemma~\ref{lem:PiJsubalg}. This proves the claim for $0\in J$.
 
 It remains to consider the case $0\not\in J$. Define $\overline{J}\coloneqq \{0\}\cup J$ and apply the construction for the case $0\in \overline{J}$
 to obtain the diagram \eqref{eqn:algebrahomosJ} for $\overline{J}$. To define the quiver $Q_J$, we remove from $Q_{\overline{J}}$ the vertex 0 together with all arrows in $Q_{\overline{J}}$ that have head and/or tail at 0. Notice that for each $i, j\in J\subset \overline{J}$, the quiver $Q_{\overline{J}}$ already has $d_{i,j}$ arrows from $i$ to $j$ corresponding to a set of $R_0$-module generators of $\Hom_{R_0}(R_i,R_j)$,
 so the desired $\CC$-algebra epimorphism $\alpha_J$ is obtained by restriction from the map $\alpha_{\overline{J}}$ in the diagram \eqref{eqn:algebrahomosJ} for $\overline{J}$. 
 
 Next, for any $j\in J$, let $\{a_{j,m}^\prime \mid 1\leq m\leq d_{0,j}\}$ denote the arrows in $Q_{\overline J}$
corresponding to our chosen set of generators of 
$\Hom_{R_0}(R_0,R_j)$. Define the quiver $Q_J^*$ to have vertex set $\{\infty\}\cup J$ and arrow set obtained by augmenting the arrow set of $Q_J$ as follows: for each $j\in J$, introduce arrows $\{a_{j,m} \mid 1\leq m\leq d_{0,j}\}$ from $\infty$ to $j$.  Note that $\beta_{\overline{J}}(a_{j,m}^\prime b)\in e_jA_{\overline{J}} e_\infty = e_j A_J e_\infty\subset A_J$. Therefore, if for any path $p$ in $\CC Q_J^*$ we define 
 \[
 \beta_J(p) = \left\{
 \begin{array}{cr} 
 \alpha_J(p) & \text{for }p\in \CC Q_J; \\
 e_{\infty} & \text{for }p=e_\infty; \\ 
 \beta_{\overline{J}}(a_{j,m}^\prime b) & \text{for }p=a_{j,m},
 \end{array}\right. 
 \]
 then we determine uniquely a $\CC$-algebra homomorphism $\beta_J\colon \CC Q_J^*\to A_J$. To show that $\beta_J$ is surjective, it suffices to check that the image of $\beta_J$ contains $A_Je_\infty$ because $\alpha_J$ is surjective.
 For this, consider $\gamma\in A_J e_\infty\subseteq e_JA_{\overline{J}}e_\infty$. Since $\beta_{\overline{J}}$ is surjective, $\gamma$~can be represented by a linear combination of paths $\gamma_i \in Q^*_{\overline{J}}$, each with tail at $\infty$ and head at a vertex in $J$.
Every such path $\gamma_i$ necessarily begins by traversing a path of the form $a_{j,m}^\prime c b$ for some $j\in J$ and $1\leq m\leq d_{0,j}$, where $c$ is a (possibly empty)
 composition of loops at vertex $0$. Crucially, as $\beta_{\overline{J}}(c) \in \Hom_{R_0}(R_0,R_0)\cong R_0$, the element $\beta_{\overline{J}}(a_{j,m}^\prime c)\in e_jA_{\overline{J}} e_0 = \Hom_{R_0}(R_0,R_j)$ can be written as the image under $\beta_{\overline{J}}$ of a linear combination $\sum_{1\leq n\leq d_{0,j}} c_n a_{j,n}^\prime$ for some $c_n\in R_0\subseteq \Hom_{R_0}(R_j,R_j)$. Therefore the start $a_{j,m}^\prime c b$ of the path $\gamma_i$ satisfies
 \[
 \beta_{\overline{J}}(a_{j,m}^\prime cb) =
 \beta_{\overline{J}}\bigg(
 \sum_{1\leq n\leq d_{0,j}} c_n a_{j,n}^\prime
 \bigg) \beta_{\overline{J}}(b) =
 \sum_{1\leq n\leq d_{0,j}} c_n \beta_{\overline{J}} (a_{j,n}^\prime b)
 =
 \beta_{J} \bigg(\sum_{1\leq n\leq d_{0,j}} \ell_n a_{j,n}\bigg),
 \]
where each $\ell_n$ is a linear combination of loops in $Q_J^*$ at vertex $j$ satisfying $\alpha_J(\ell_n) = c_n$. 
Thus, the image in $A_J$ of the beginning of our path $\gamma_i$ lies in the image of $\beta_J$.  It follows that each path $\gamma_i$ arising in the linear combination of $\gamma$ lies in the image of $\beta_J$, because $\alpha_J$ is surjective. Therefore the image of $\beta_J$ contains $A_Je_\infty$ as required.
 \end{proof}

  \begin{remarks}
 \begin{enumerate}
     \item The quiver $Q_J^*$ that we construct in the proof above has many more arrows than necessary. For example, when $J=\{0,1,\dots,r\}$, the algorithm returns a quiver with arrow set containing at least three loops at each vertex, whereas $Q^*$ contains no loops.
     \item An alternative proof for Proposition~\ref{prop:quiver} could be given by exhibiting a finite number of paths in $Q^*$ whose tail and head both lie in the set $\{\infty\}\cup J$, with the property that their classes, up to the preprojective relations, generate the cornered algebra $A_J$. While we believe this is indeed possible, the combinatorics of the situation gets rather intricate, especially in the case $0\notin J$. The proof presented above has the advantage that it avoids case-by-case analysis of Dynkin diagrams.
 \end{enumerate}
 \end{remarks}
 
 \section{Reconstructing quiver varieties via the cornered algebras}
   In general, the quiver variety $\mathfrak{M}_{\theta_J}$ is the coarse moduli space for S-equivalence classes of $\theta_J$-semistable $\Pi$-modules of dimension vector $v$. However, in the special case $J= \{0,\dots,r\}$ it may also be regarded as the fine moduli space of isomorphism classes of $\theta_J$-stable $A$-modules of dimension vector $v$ by Lemma~\ref{lem:PiA}. 
   We now introduce an alternative, fine moduli space construction for each $\mathfrak{M}_{\theta_J}$ using the algebra $A_J$.

The element
 \[
 v_J\coloneqq\rho_\infty+\sum_{j\in J} n\dim(\rho_j)\rho_j\in \ZZ\oplus \ZZ^J
 \]
 is a dimension vector for $A_J$-modules, and we consider the stability condition $\eta_J\colon \ZZ\oplus \ZZ^J\to \mathbb{Q}$ given by 
 \[
 \eta_J(\rho_i)  = \left\{\begin{array}{cr} -\sum_{j\in J}n\dim(\rho_j) & \text{ for } i=\infty \\ 1 & \text{ if }i\in J \end{array}\right.
 \]
 It follows directly from the definition that an $A_J$-module $N$ of dimension vector $v_J$ is $\eta_J$-stable if and only if there exists a surjective $A_J$-module homomorphism $A_J e_\infty \to N$. The vector $v_J$ is indivisible and $\eta_J$ is a generic stability condition for $A_J$-modules. 

The quiver moduli space construction of King~\cite[Proposition~5.3]{King94} for finite-dimensional algebras can be adapted to any algebra presented as the quotient of a finite connected quiver by an ideal of relations. Thus,  Proposition~\ref{prop:quiver} allows us to define the fine moduli space $\mathcal{M}(A_J)$ of $\eta_J$-stable $A_J$-modules of dimension vector $v_J$. Let $T_J\coloneqq\bigoplus_{i\in \{\infty\}\cup J}T_i$ denote the tautological bundle on $\mathcal{M}(A_J)$, where $T_\infty$ is the trivial bundle and $T_j$ has rank $n\dim(\rho_j)$ for $j\in J$. The line bundle
 \[
 \mathcal{L}_J\coloneqq\bigotimes_{j\in J}\det(T_j)
 \]
 is the polarising ample bundle on $\mathcal{M}(A_J)$ given by the GIT construction. 
 
  \begin{lem}
  \label{lem:tauJ}
  Let $\theta\in C_+$, and let $J\subseteq \{0,\dots,r\}$ be any non-empty subset. There is a universal morphism 
 \begin{equation}
 \label{eqn:tauJ}
 \tau_J\colon \mathfrak{M}_\theta\to \mathcal{M}(A_J)
 \end{equation}
 satisfying  $\tau_J^*(T_i) \cong \mathcal{R}_i$ for $i\in \{\infty\}\cup J$. 
 \end{lem}
 \begin{proof}
 In light of the universal property of $\mathcal{M}(A_J)$, it suffices to show that the locally-free sheaf 
   \[
   \mathcal{R}_J\coloneqq \bigoplus_{i\in \{\infty\}\cup J} \mathcal{R}_i
   \]
 of rank $1+\sum_{j\in J} n\dim(\rho_j)$ on 
 the quiver variety $\mathfrak{M}_\theta$ is a flat family of $\eta_J$-stable $A_J$-modules of dimension vector $v_J$. We saw in Lemma~\ref{lem:PiA} that we may delete one arrow from $Q$, giving rise to a $\CC$-algebra homomorphism $\phi\colon A\to \End(\mathcal{R})$. Multiplying this on the left and right by the idempotent $e_J$ determines a $\CC$-algebra homomorphism $A_J\to \End(\mathcal{R}_J)$ which makes $\mathcal{R}_J$ into a flat family of $A_J$-modules of dimension vector $v_J$. To establish stability, write $\bigoplus_{i\in Q_0} \mathcal{R}_{i,y}$ for the fibre of $\mathcal{R}$ over a closed point $y\in \mathfrak{M}_\theta$. The fact that $\bigoplus_{i\in Q_0} \mathcal{R}_{i,y}$ is $\theta$-stable is equivalent to the existence of a surjective $A$-module homomorphism $A e_\infty\to \bigoplus_{i\in Q_0} \mathcal{R}_{i,y}$. Applying $e_J$ on the left produces a surjective $A_J$-module homomorphism $A_J e_\infty\to \bigoplus_{i\in \{\infty\}\cup J} \mathcal{R}_{i,y}$ which in turn is equivalent to $\eta_J$-stability of the fibre $\bigoplus_{i\in \{\infty\}\cup J} \mathcal{R}_{i,y}$ of $\mathcal{R}_J$ over $y\in \mathfrak{M}_\theta$. In particular, $\mathcal{R}_J$ is a flat family of $\eta_J$-stable $A_J$-modules of dimension vector $v_J$.
 \end{proof}
 
 \begin{remarks}
 \begin{enumerate}
 \item An alternative proof of Lemma~\ref{lem:tauJ} uses the fact that the tautological bundles $\mathcal{R}_i$ on $\mathfrak{M}_\theta$ are globally generated for $i\in I$ by \cite[Corollary~2.4]{CIK18}, in which case one can adapt the proof of \cite[Proposition~2.3]{CIK18} to deduce that $\mathcal{R}_J$ is a flat family of $\eta_J$-stable $A_J$-modules of dimension vector $v_J$. In particular, global generation is the key feature in Lemma~\ref{lem:tauJ}, just as in the proof of Lemma~\ref{lem:ampleVGIT}. This is not a coincidence; see Theorem~\ref{thm:coarsetofine}.
 \item Building on Remark~\ref{rem:highmultiple}, we now take an even higher multiple of $\theta$ if necessary (and the same high multiple of each $\eta_J$ and each $\theta_J$) to ensure that the polarising ample line bundles on $\mathcal{M}(A_J)$ and on $\mathfrak{M}_{\theta_J}$ are very ample for all relevant $J\subseteq \{0,\dots, r\}$.
 \end{enumerate}
 \end{remarks}

 \begin{lem}
 \label{lem:5varieties}
  Let $\theta\in C_+$ and assume $J\subseteq \{0,\dots,r\}$ is non-empty. There is a commutative diagram
   \begin{equation}
 \label{eqn:5varieties}
\begin{tikzcd}
 & \mathfrak{M}_\theta\ar[dl,swap,"{\pi_J}"] \ar[dr,"{\tau_J}"]\ar[dd,"\varphi_{\vert L_J\vert}"] & \\   \mathfrak{M}_{\theta_J}\ar[rd,hook] & & \mathcal{M}(A_J)\ar[d,hook,"\varphi_{\vert \mathcal{L}_J\vert}"] \\
 & \vert L_J\vert \ar[r,"\psi"]& \vert \mathcal{L}_J\vert
 \end{tikzcd}
\end{equation}
  of schemes over $\Sym^n(\mathbb{C}^2/\Gamma)$,  where $\psi$ is an isomorphism. 
   \end{lem}
 \begin{proof}
  The commutative triangle on the left of \eqref{eqn:5varieties} was constructed in Lemma~\ref{lem:ampleVGIT}. For the quadrilateral on the right, our choice of $\eta_J$ ensures that the polarising line bundle $\mathcal{L}_J$ on $\mathcal{M}(A_J)$ is very ample, so the morphism $\varphi_{\vert \mathcal{L}_J\vert}$ is well-defined. Since pullback commutes with tensor operations on the $T_i$,  the isomorphisms $\tau_J^*(T_i) \cong \mathcal{R}_i$ for $i\in J$ imply that $L_J = \tau_J^*(\mathcal{L}_J)$. If $\mathcal{O}_{\vert \mathcal{L}_J\vert}(1)$ denotes the polarising ample bundle on $\vert \mathcal{L}_J\vert$, then
  \begin{equation}
      \label{eqn:quadmorphisms}
  (\varphi_{\vert \mathcal{L}_J\vert}\circ \tau_J)^*(\mathcal{O}_{\vert \mathcal{L}_J\vert}(1)) = \tau_J^*(\mathcal{L}_J) = L_J = \varphi^*_{\vert L_J \vert}\big(\mathcal{O}_{\vert L_J\vert}(1)\big)
  \end{equation}
 on $\mathfrak{M}_\theta$.  The morphism to a complete linear series is unique up to an automorphism of the linear series, so there is an isomorphism $\psi\colon\vert L_J\vert\to\vert \mathcal{L}_J\vert$ such that $\varphi_{\vert \mathcal{L}_J\vert}\circ \tau_J = \psi\circ \varphi_{\vert L_J\vert}$ as required. 
 
 It remains to show that \eqref{eqn:5varieties} is a diagram of schemes over $\Sym^n(\CC^{2}/\Gamma)$. The Leray spectral sequence for the resolution $\pi\colon \mathfrak{M}_\theta\to \mathfrak{M}_0\cong \Sym^n(\CC^{2}/\Gamma)$ gives $H^0(\mathcal{O}_{\mathcal{M}_\theta})\cong H^0(\mathcal{O}_{\mathcal{M}_0})\cong (\CC[V]^\Gamma)^{\mathfrak{S}_n}$ because $\Sym^n(\CC^{2}/\Gamma)$ has rational singularities. It follows that $\pi = \varphi_{\vert \mathcal{O}_{\mathfrak{M}_\theta}\vert}$, i.e. $\pi$ is the structure morphism of $\mathfrak{M}_\theta$ as a variety over $\Sym^n(\CC^{2}/\Gamma)$. Repeating the argument from \eqref{eqn:quadmorphisms}, with the roles of $L_J, \mathcal{L}_J$ and $\mathcal{O}_{\vert \mathcal{L}_J\vert}(1)$ played instead by the trivial bundles on $\mathfrak{M}_\theta$, $\mathcal{M}(A_J)$ and $\Sym^n(\CC^{2}/\Gamma)$ respectively, shows that $\mathcal{M}(A_J)$ is a scheme over $\Sym^n(\CC^{2}/\Gamma)$. It follows that \eqref{eqn:5varieties} is a diagram of schemes over $\Sym^n(\CC^{2}/\Gamma)$.
 \end{proof}
 
   Our goal for the rest of this section is to add a morphism
  $\iota_J\colon \mathfrak{M}_{\theta_J}\to \mathcal{M}(A_J)$ to diagram \eqref{eqn:5varieties} and to show that $\iota_J$ is an isomorphism on the underlying reduced schemes. Consider the functors
  \[
\begin{tikzpicture}
\node (v1) at (0,0)  {$\scriptsize{A\module}$};
\node (v2) at (4,0)  {$\scriptsize{A_J\module}$};

\draw [->,bend right=-5] (v1) to node[gap] {$\scriptstyle{j^*}$} (v2);
\draw [->,bend left=5] (v2) to node[gap] {$\scriptstyle{j_{!}}$} (v1);
\end{tikzpicture}
\]
  defined by $j^*(-)\coloneqq e_JA\otimes_A (-)$ and $j_!(-)\coloneqq  A e_J\otimes_{A_J}(-)$. These are two of the six functors in a recollement of the module category $A\module$ \cite{FP04}. In particular, $j^*$ is exact, $j^*j_!$ is the identity functor, and for any $A_J$-module $N$, the $A$-module $j_!(N)$ is the maximal extension by $A/(A e_JA)$-modules; see \cite[(3.4)]{CIK18}.

 \begin{lem}
 \label{lem:verbatim}
  Let $N$ be an $A_J$-module of dimension vector $v_J$.
  \begin{enumerate}
      \item[\one] If there exists a surjective $A_J$-module homomorphism $A_Je_\infty\to N$, then there exists a surjective $A$-module homomorphism $Ae_\infty \to j_!(N)$.
      \item[\two] The $A$-module $j_!(N)$ is finite-dimensional and satisfies $\dim_i j_!(N) = \dim_i N $ for all $i\in  \{\infty\}\cup J$.
  \end{enumerate} 
  \end{lem}
  \begin{proof}
  The construction of the quiver from Proposition~\ref{prop:quiver} shows that $A$ is a finitely generated module over the algebra $R_0 \cong e_0 A e_0$. Armed with this observation, the proof of \cite[Lemma~3.6]{CIK18} applies verbatim (the notation differs slightly for part \one: our map $A_Je_\infty\to N$ is written $A_Ce_0^\prime\to N$ in \emph{ibid.}).
  \end{proof}

  \begin{lem}
  \label{lem:recollement}
  Let $N$ be an $\eta_J$-stable $A_J$-module of dimension vector $v_J$. The $A$-module $j_!(N)$ is $\theta_J$-semistable.
  \end{lem}
   \begin{proof}
 Since $N$ is $\eta_J$-stable, there is a surjective $A_J$-module homomorphism $A_J e_\infty \to N$. Lemma~\ref{lem:verbatim} 
gives a surjective $A$-module homomorphism $A e_\infty \to j_!(N)$ and, moreover, 
 the finite dimensional $A$-module $j_!(N)$ satisfies $\dim_i j_!(N) = \dim_i N$ for $i\in \{\infty\}\cup J$. 
 Recall that  $\theta_J(\rho_i) = 0$ for $i\not\in \{\infty\}\cup J$, so 
  \[
  \theta_J\big(j_!(N)\big) = \theta_J\left(\sum_{i\in \{\infty\}\cup J} \dim_i(j_!(N))\rho_i\right) = \eta_J\left(\sum_{i\in \{\infty\}\cup J} \dim_i(N)\rho_i\right) = \eta_J(N) = 0.
  \]
Now let $M\subset j_!(N)$ be a proper submodule. If $\dim_\infty M=1$, then surjectivity of the map $A e_\infty \to j_!(N)$ gives $M=j_!(N)$ which is absurd, so $\dim_\infty M=0$. But $\theta_J(\rho_i)\geq 0$ for all $i\neq \infty$, so $\theta_J(M)\geq 0$ as required.
\end{proof}

  \begin{lem}
  \label{lem:existsemistable}
 Let $N$ be an $\eta_J$-stable $A_J$-module of dimension vector $v_J$.  Then there exists a $\theta_J$-semistable $A$-module $M$ such that $j^*M \cong N$ and $\dim_i M \leq  n\dim(\rho_i)$ for all $i\not\in \{\infty\}\cup J$.
    \end{lem}
    \begin{proof} By Lemma~\ref{lem:recollement}, $j_!(N)$ is $\theta_J$-semistable. If $\dim_i j_!(N) \leq  n\dim(\rho_i)$ for $i\not\in \{\infty\}\cup J$, then we can simply set $M\coloneqq j_!(N)$, as $j^*j_!$ is the identity. Otherwise, consider the $\theta_J$-polystable module $\bigoplus_{\lambda} M_\lambda$ that is $S$-equivalent to $j_!(N)$. Let $M_{\lambda_\infty}$ denote the unique summand satisfying $\dim_\infty M_{\lambda_\infty} =1$. Since $M_{\lambda_\infty}$ is by construction a $\theta_J$-stable $A$-module, it follows that $\dim_i M_{\lambda_\infty} =n\dim(\rho_i)$ for all $i\in J$, and hence $\dim_i M_\lambda =0$ for all $\lambda\neq {\lambda_\infty}$ and all $i\in \{\infty\}\cup J$. For each index $\lambda$ and for all $i\in \{\infty\}\cup J$, we have
\[
\dim_i j^*M_\lambda = \dim e_i \big(e_J A\otimes_{A}(M_\lambda)\big) = \dim e_i A\otimes_A M_\lambda = \dim_i M_\lambda.
\]
 It follows that $\dim_i j^*M_\lambda = 0$ for all $\lambda\neq {\lambda_\infty}$ and $i\in \{\infty\}\cup J$, and hence
 $j^*M_\lambda=0$ for $\lambda\neq {\lambda_\infty}$.

  We claim that $j^*M_{\lambda_\infty}$ is isomorphic to $N$. Indeed, the $A$-module $j_!(N)$ is $\theta_J$-semistable by Lemma~\ref{lem:recollement}, and the $\theta_J$-stable $A$-modules $M_\lambda$ are by construction the factors in the composition series of $j_!(N)$ in the category of $\theta_J$-semistable $A$-modules. It follows from exactness of $j^*$ that the $A_J$-modules $j^*M_\lambda$ are the factors in the composition series of $j^*j_!(N)\cong N$ in the category of $\eta_J$-semistable $A_J$-modules. But $j^*M_\lambda=0$ for $\lambda\neq {\lambda_\infty}$, so the only nonzero factor of the composition series is $j^*M_{\lambda_\infty}$. It follows that $j^*M_{\lambda_\infty}\cong N$, because the factor $j^*M_{\lambda_\infty}$ can only appear once in the composition series.
  
  As a result, the $\theta_J$-stable $A$-module $M_{\lambda_\infty}$ satisfies  $j^*M_{\lambda_\infty}\cong N$ and $\dim_i M_{\lambda_\infty} =n\dim(\rho_i)$ for all $i\in J$. Therefore $M_{\lambda_\infty}$ is the required $A_J$-module as long as $\dim_i M_{\lambda_\infty} \leq  n\dim(\rho_i)$ for $i\not\in \{\infty\}\cup J$. We establish this key inequality in Appendix~\ref{sec:the ugly proof}.
\end{proof}

\begin{remark}\label{rem:leftwithDynkin} The modules $M_\lambda$ for $\lambda\ne \lambda_\infty$ in the proof of Lemma~\ref{lem:existsemistable} are in fact all 1-dimensional vertex simples. 
 To see this, note that removing any non-empty set of vertices and their incident edges from an extended Dynkin diagram gives a diagram in which every connected component is Dynkin of finite type. 
 Thus removing the vertices $\{\infty\}\cup J$ and all incident edges from the framed extended diagram leaves us with a collection of Dynkin diagrams of finite type. Choose $\lambda\ne \lambda_\infty$. Since $\dim_j M_\lambda =0$ for all $j\in \{\infty\}\cup J$, $M_\lambda$ is a simple module of the preprojective algebra of a quiver of finite type. But such modules are one-dimensional by \cite[Lemma~2.2]{SavTing}.
 \end{remark}

 \begin{thm}
 \label{thm:coarsetofine}
 For any non-empty $J\subseteq \{0,\dots,r\}$, there is a commutative diagram of morphisms
 \begin{equation}
 \label{eqn:isomDiagram}
 \begin{tikzcd}
 & \mathfrak{M}_\theta\ar[dl,swap,"{\pi_J}"] \ar[dr,"{\tau_J}"] & \\
  \mathfrak{M}_{\theta_J}\ar[rr,"{\iota_J}"] & & \mathcal{M}(A_J),
 \end{tikzcd}
\end{equation}
  where $\iota_J$ is an isomorphism of the underlying reduced schemes. In particular, $\mathcal{M}(A_J)$ is irreducible, and its underlying reduced scheme is normal and has symplectic singularities.
 \end{thm}
 
 \begin{proof}
 Let $\sigma_J\colon \mathfrak{M}_{\theta_J}\to \vert \mathcal{L}_J\vert$ be the composition of the isomorphism $\psi$ of Lemma~\ref{lem:5varieties} with the closed immersion $\mathfrak{M}_{\theta_J}\hookrightarrow \vert L_J\vert$ from diagram \eqref{eqn:5varieties}. Since $\sigma_J$ is a closed immersion, it identifies $\mathfrak{M}_{\theta_J}$ with $\im(\sigma_J)$. Surjectivity of $\pi_J$ and commutativity of diagram \eqref{eqn:5varieties} then imply that $\mathfrak{M}_{\theta_J}$ is isomorphic to the subscheme $\im(\sigma_J\circ \pi_J) = \im(\varphi_{\vert \mathcal{L}_J\vert}\circ \tau_J)$ of $\vert \mathcal{L}_J\vert$.  Since $\mathcal{L}_J$ is the polarising very ample line bundle on the GIT quotient $\mathcal{M}(A_J)$, the closed immersion $\varphi_{\vert \mathcal{L}_J\vert}$ induces an isomorphism $\lambda_J\colon \im(\varphi_{\vert \mathcal{L}_J\vert})\to \mathcal{M}(A_J)$. The morphism
 \[
 \iota_J\coloneqq \lambda_J\circ \sigma_J\colon \mathfrak{M}_{\theta_J}\longrightarrow \mathcal{M}(A_J)
 \]
 is therefore a closed immersion. Note that 
 \[
 \iota_J\circ \pi_J = \lambda_J\circ \sigma_J\circ \pi_J = \lambda_J \circ \varphi_{\vert \mathcal{L}_J\vert}\circ \tau_J = \tau_J,
 \]
 so diagram \eqref{eqn:isomDiagram} commutes. In order to prove that  $\iota_J$ is an isomorphism of the underlying reduced schemes, it suffices to show that $\iota_J$ is surjective on closed points. 
 
 Consider a closed point $[N]\in \mathcal{M}(A_J)$, where $N$ is an $\eta_J$-stable $A_J$-module of dimension vector $v_J$. Let $M$ be the $\theta_J$-semistable $A$-module from Lemma~\ref{lem:existsemistable}.  For $i\not\in \{\infty\}\cup J$, define $m_i\coloneqq n\dim(\rho_i)-\dim_i M\geq 0$ and let $S_i\coloneqq\CC e_i$ denote the vertex simple $A$-module at vertex $i\in Q_0$. The $A$-module
 \[
 \overline{M}\coloneqq M\oplus \bigoplus_{i\in \{0,\dots,r\}\setminus J} S_i^{\oplus m_i}
 \]
 is $\theta_J$-semistable of dimension vector $v$ by construction, and it satisfies $j^*(\overline{M}) = j^*(M) = N$. Write $[\overline{M}]\in \mathfrak{M}_{\theta_J}$ for the corresponding closed point, and let $\widetilde{M}$ be any $\theta$-stable $A$-module of dimension vector $v$ such that the closed point $[\widetilde{M}]\in \mathfrak{M}_\theta$ satisfies $\pi_J([\widetilde{M}])=[\overline{M}]\in \mathfrak{M}_{\theta_J}$. Then $j^*(\widetilde{M})=j^*(\overline{M})=N$, hence $\tau_J([\widetilde{M}])= [N]$, and commutativity of diagram \eqref{eqn:isomDiagram} gives that
 \[
 \iota_J([\overline{M}]) = (\iota_J\circ \pi_J)\big([\widetilde{M}]\big) = \tau_J\big([\widetilde{M}]\big) = [N],
 \]
 so $\iota_J$ is indeed surjective. The final statement of Theorem~\ref{thm:coarsetofine} follows from Lemma~\ref{lem:Mthetageometry} and Lemma~\ref{lem:5varieties}.
 \end{proof}
 
 \begin{remarks}
 \label{rem:Jnonempty}

 \begin{enumerate}
  \item  If $J\neq \{0,\dots,r\}$, then the stability parameter $\theta_J$ lies in the boundary of the GIT chamber $C_+$, so $\mathfrak{M}_{\theta_J}$ does not admit a universal family of $\theta_J$-semistable $\Pi$-modules of dimension vector $v$. However, the fine moduli space $\mathcal{M}(A_J)$ does carry a universal family $T_J$ of  $\eta_J$-stable $A_J$-modules of dimension vector $v_J$, and hence under the isomorphism of Theorem~\ref{thm:coarsetofine}, the bundle $\iota_J^*(T_J)$ on $\mathfrak{M}_{\theta_J}$ pulls back along $\pi_J$ to the summand $\bigoplus_{i\in \{\infty\}\cup J} \mathcal{R}_i$ of the tautological bundle on $\mathfrak{M}_\theta$.
     \item In the course of the proof of Theorem~\ref{thm:coarsetofine}, we deduce directly that $\tau_J$ is surjective on closed points.
     \item For $J=\varnothing$, we have $\mathfrak{M}_{\theta_J} \cong \Sym^n(\CC^{2}/\Gamma)$. On the other hand, $e_\infty A e_\infty = \CC e_\infty$, which does not provide enough information with which to reconstruct $\Sym^n(\CC^{2}/\Gamma)$.
 \end{enumerate}
 \end{remarks}

\section{Identifying the posets for the coarse and fine moduli problems}
 We now establish that the morphisms $\iota_J\colon \mathfrak{M}_{\theta_J}\to \mathcal{M}(A_J)$ from Theorem~\ref{thm:coarsetofine} are compatible with the morphisms $\pi_{J,J'}\colon \mathfrak{M}_{\theta_J}\to \mathfrak{M}_{\theta_{J'}}$ that feature in the poset introduced in Proposition~\ref{prop:poset}.

   \begin{lem}
   \label{lem:functorial}
   For non-empty subsets $J^\prime\subset J\subset \{0,1,\dots,r\}$, there is a commutative diagram
   \begin{equation}
   \label{eqn:pitaudiagram}
\begin{tikzcd}
 \mathfrak{M}_{\theta_J} \ar[r,"\iota_J"] \ar[d,swap,"{\pi_{J,J'}}"]& \mathcal{M}(A_J)\ar[d,"{\tau_{J,J'}}"]   \\
  \mathfrak{M}_{\theta_{J'}}\ar[r,"\iota_{J'}"] & \mathcal{M}(A_{J'}) 
 \end{tikzcd}
\end{equation}
in which the horizontal arrows are isomorphisms on the underlying reduced schemes and the vertical arrows are surjective, projective, birational morphisms. 
   \end{lem}
    \begin{proof}
   The subbundle  $\bigoplus_{i\in \{\infty\}\cup J'} T_i$ of the tautological bundle $T_J$ on $\mathcal{M}(A_J)$ is a flat family of $\eta_{J'}$-stable $A_{J'}$-modules of dimension vector $v_{J'}$, so there is a universal morphism
   \[
   \tau_{J,J'}\colon \mathcal{M}(A_J)\longrightarrow \mathcal{M}(A_{J'})
 \]
 satisfying $\tau_{J,J'}^*(T_i') = T_i$ for $i\in \{\infty\}\cup J'$, where $\bigoplus_{i\in \{\infty\}\cup J'} T_i^\prime$ is the tautological bundle on $\mathcal{M}(A_{J'})$. Now 
 \[
 \big( \tau_{J,J'}\circ \tau_J\big)^*(T_i') = \tau_J^*(T_i) = \mathcal{R}_i = \tau_{J'}^*(T_i')
 \]
 for all $i\in \{\infty\}\cup J'$, and since this property characterises the morphism $\tau_{J'}$, we have a commutative diagram
      \begin{equation}
   \label{eqn:tautaudiagram}
\begin{tikzcd}
 & \mathfrak{M}_\theta \ar[dl,swap,"\tau_J"] \ar[dr,"\tau_J"] & \\
 \mathcal{M}(A_J) \ar[rr,"\tau_{J,J'}"] & & \mathcal{M}(A_{J'}).
 \end{tikzcd}
\end{equation}
Proposition~\ref{prop:poset} gives a similar commutative diagram expressing the identity $\pi_{J,J'}\circ \pi_J = \pi_{J'}$ for morphisms between quiver varieties, while Theorem~\ref{thm:coarsetofine} establishes the identities $\iota_J\circ \pi_J=\tau_J$ and $\iota_{J'}\circ \pi_{J'}=\tau_{J'}$. Taken together, these identities show that the maps in all four triangles in the following pyramid diagram commute:
     \begin{equation}
   \label{eqn:pyramidDiagram}
\begin{tikzcd}
 & \mathfrak{M}_\theta \ar[dl,swap,"\pi_J"]    \ar[dr,"\tau_J"] & \\
 \mathfrak{M}_{\theta_J} \ar[rr, "\iota_J"] \ar[dd,swap,"{\pi_{J,J'}}"]& & \mathcal{M}(A_J)\ar[dd,"{\tau_{J,J'}}"]   \\
  & & \\
  \mathfrak{M}_{\theta_{J'}}\ar[rr,"\iota_{J'}"]\arrow[from=uuur,"\pi_{J'}",crossing over] & & \mathcal{M}(A_{J'}).\arrow[from=uuul,swap,crossing over,"\tau_{J'}"]
 \end{tikzcd}
\end{equation}
To show that the morphisms around the pyramid's square base commute, choose for any closed point $x\in \mathfrak{M}_{\theta_J}$ a lift $y\in \pi_J^{-1}(x)\subset \mathfrak{M}_\theta$. Commutativity of the triangles in the diagram gives
\[
\big(\iota_{J'}\circ \pi_{J,J'}\big)(x) = \iota_{J'}(\pi_{J'}(y)\big)  = \tau_{J'}(y) = \tau_{J,J'}(\tau_J(y)\big) = \big(\tau_{J,J'}\circ \iota_J\big)(x),
  \]
and since $x\in \mathfrak{M}_\theta$ was arbitrary and $\pi_J$ is surjective, we have that $\iota_{J'}\circ \pi_{J,J'} = \tau_{J,J'}\circ \iota_J$ as required.
   \end{proof}

We deduce the following.

  \begin{prop}
  \label{prop:poset2}
  The face poset of the cone $\overline{C_+}$ can be identified with the poset on the set of fine moduli spaces $\mathcal{M}(A_J)$ for non-empty subsets $J\subseteq \{0,\dots,r\}$ together with $\CC^{2n}/\Gamma_n$, where edges in the Hasse diagram of the poset indicating inequalities $\mathfrak{M}(A_J) > \mathfrak{M}(A_{J'})$ and $\mathfrak{M}(A_J) > \CC^{2n}/\Gamma_n$ are realised by the universal morphisms $\tau_{J,J'}$ and the structure morphisms $\varphi_{\vert \mathcal{O}_{\mathcal{M}(A_J)}\vert}$ respectively.
  \end{prop}
 
 \section{Punctual Hilbert schemes for Kleinian singularities}\label{sec:HilbSchemesOfSings}
 In this section, we specialise to the case $J=\{0\}$ and study the algebra $A_J$, before establishing the link between the fine moduli space $\mathcal{M}(A_J)$ and the Hilbert scheme of $n$ points on $\CC^2/\Gamma$. It will be convenient to write dimension vectors of $A_J$-modules as pairs $(v_\infty,v_0)$ in this case.
 
As we saw in Proposition~\ref{prop:quiver}, the algebra $A_J$ can be presented as the path algebra of a quiver modulo an ideal of relations. The relations appear to be fairly complicated, but for $J=\{0\}$ it is possible to give an explicit presentation of $A_J$;
 this will turn out to be sufficient for our purposes.
 To spell this out, recall the construction of the quiver $Q_J^*$ from Proposition~\ref{prop:quiver} that has vertex set
 $\{\infty, 0\}$ and arrow set 
comprising one arrow 
$b$ from $\infty$ to $0$, and loops $\alpha_1, \alpha_2, \alpha_3$ at vertex $0$ corresponding to a set of minimal $\CC$-algebra generators of $\Hom_{R_0}(R_0,R_0)= R_0 \cong \CC[x,y]^\Gamma$ as shown in Figure~\ref{fig:quiverQ'}:
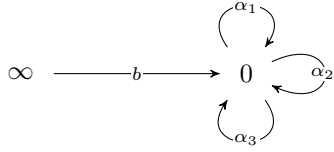
\begin{figure}[h!]
\centering
\begin{tikzpicture}[yscale=0.77]
\node (M0) at (-1,0) {$\infty$};
\node (M1) at (2,0) {$0$};
\draw [->] (M0) to node[gap] {$\scriptstyle{b}$} (M1);
\draw [->,looseness=12, out=120, in=60] (M1) to node[gap] {$\scriptstyle{\alpha_1}$} (M1);
\draw [->,looseness=10, out=30, in=-30] (M1) to node[gap] {$\scriptstyle{\alpha_2}$} (M1);
\draw [->,looseness=12, out=-60, in=-120] (M1) to node[gap] {$\scriptstyle{\alpha_3}$} (M1);
\end{tikzpicture}
\caption{The quiver $Q_J^*$ used in the presentation of $A_J$ for $J=\{0\}$.}
\label{fig:quiverQ'}
\end{figure}

 To state the presentation of $A_J$ in this case, recall the presentation of the algebra $\CC[x,y]^\Gamma$ from \eqref{eqn:duValEquation}.

  \begin{lem}\label{lem:presentationPi_J}
  For $J = \{0\}$, 
 the algebra $A_J$ is isomorphic to the quotient of $\CC Q_J^*$ by the two-sided ideal
 \begin{equation}
  \label{eqn:K2}
 K= \big(f(\alpha_1,\alpha_2, \alpha_3), \alpha_1\alpha_2-\alpha_2\alpha_1, \alpha_1\alpha_3-\alpha_3\alpha_1, \alpha_2\alpha_3-\alpha_3\alpha_2\big),
  \end{equation}
where $f\in\CC[z_1,z_2,z_3]$ is the defining equation of the hypersurface $\CC^2/\Gamma\subseteq \Spec \CC[z_1,z_2,z_3]$. 
 \end{lem}
 
 \begin{proof}
 In Proposition~\ref{prop:quiver}, we constructed a $\CC$-algebra epimorphism $\beta_J\colon \CC Q_J^*\to A_J$. The images under $\beta_J$ of the arrows $\alpha_1, \alpha_2, \alpha_3$ correspond to minimal $\CC$-algebra generators of $R_0\cong e_0 A e_0$, so the ideal $K$ lies in the kernel of $\beta_J$.  We claim that the induced $\CC$-algebra epimorphism
 \[
\beta_J\colon \CC Q_J^*/K \longrightarrow A_J
\]
is an isomorphism. Define a $\CC$-algebra homomorphism $\gamma_J\colon A_J\to \CC Q_J^*/K$ by sending the chosen minimal $\CC$-algebra generators of $R_0\cong e_0 A_J e_0$ to the classes of the arrows $\alpha_1, \alpha_2, \alpha_3$ in $\CC Q_J^*/K$, and by sending the classes of $e_\infty$ and $b$ in $A_J$ to the classes of the paths $e_\infty$ and $b$ in $\CC Q_J^*/K$. This defines a $\CC$-algebra homomorphism,  because $e_0 A_J e_0$ is a subalgebra of $A$ with quotient $A/(e_\infty)$. Clearly $\gamma_J = \beta_J^{-1}$ as required.
 \end{proof}

 \begin{prop}
 \label{prop:HilbIsom}
 For the subset $J=\{0\}$, there is an isomorphism of schemes
 \[ \omega_n\colon \Hilb^{[n]}(\mathbb{C}^2/\Gamma) \to \mathcal{M}(A_J)\]
 over $\Sym^n(\CC^2/\Gamma)$.
 \end{prop}
\begin{proof} 
We begin by constructing the morphism of schemes $\omega_n$. Let $\mathcal{T}$ denote the tautological rank $n$ bundle on $\Hilb^{[n]}(\CC^2/\Gamma)$, and write $\mathcal{O}$ for the trivial bundle. In light of the universal property of $\mathcal{M}(A_J)$, it suffices to show that $\mathcal{O}\oplus \mathcal{T}$ carries a natural structure of a flat family of $\eta_J$-stable $A_J$-modules of dimension vector $v_J = (1,n)$ on $\Hilb^{[n]}(\CC^2/\Gamma)$. A closed point of $\Hilb^{[n]}(\mathbb{C}^2/\Gamma)$ corresponds to a codimension $n$ ideal $I\lhd \CC[x,y]^\Gamma \cong \CC[z_1,z_2,z_3]/(f)$. The quotient vector space $\CC[x,y]^\Gamma/I$ is of dimension~$n$, it carries the action of commuting arrows $\alpha_1, \alpha_2, \alpha_3$ satisfying the relation $f$, and has a distinguished generator $[1]\in\CC[x,y]^\Gamma/I$, which can be thought of as the image of a map $w$ from a one-dimensional vector space. 
Lemma~\ref{lem:presentationPi_J} now shows that we get the data of an $A_J$-module of dimension vector $(1,n)$. This module is moreover cyclic with generator at vertex $\infty$, so it is $\eta_J$-stable as required.  This construction works relatively over the whole of $\Hilb^{[n]}(\CC^2/\Gamma)$, equipping $\mathcal{O}\oplus \mathcal{T}$ with the structure of a family of $\eta_J$-stable $A_J$-modules as claimed. Moreover, since the bundle $\mathcal{O}\oplus \mathcal{T}$ inducing $\omega_n$ has $\mathcal{O}$ as a summand, and since the trivial bundle on any scheme induces the structure morphism, we see that $\omega_n$ commutes with the structure morphisms to $\Sym^n(\CC^2/\Gamma)$.  

Reading  Lemma~\ref{lem:presentationPi_J} in the opposite direction, an $\eta_J$-stable $A_J$-module of dimension vector $(1,n)$ defines a cyclic $\CC[x,y]^\Gamma$-module of dimension $n$ over $\CC$. The universal property of $\Hilb^{[n]}(\CC^2/\Gamma)$ then ensures that the flat family $T_\infty\oplus T_0$ of $\eta_J$-stable $A_J$-modules of dimension vector $(1,n)$ over $\mathcal{M}(A_J)$ determines a morphism $ \mathcal{M}(A_J)\to \Hilb^{[n]}(\CC^2/\Gamma)$, which is by construction the inverse of the morphism $\omega_n$.
\end{proof}

 We deduce Theorem~\ref{thm:mainintro} announced in the Introduction.

 \begin{cor}
 For any $n\geq 1$, the reduced scheme underlying $\Hilb^{[n]}(\mathbb{C}^2/\Gamma)$ is isomorphic to the quiver variety $\mathfrak{M}_{\theta_0}$ for the parameter $\theta_0 = (-n,1,0,\dots,0)$ (compare Figure~\ref{fig:A2n=3}). In particular, $\Hilb^{[n]}(\mathbb{C}^2/\Gamma)_{\rm red}$ is a normal, irreducible scheme over $\CC^{2n}/\Gamma_n$ with symplectic singularities that admits a unique projective symplectic resolution, namely the morphism  
 \[
 n\Gamma\text{-}\Hilb(\CC^2)\to \Hilb^{[n]}(\mathbb{C}^2/\Gamma)_{\rm red}
 \]
 that sends an ideal $I$ in $\CC[x,y]$ to the ideal $I\cap \CC[x,y]^\Gamma$.
  \end{cor}
 \begin{proof}
 The first statement follows from Theorem~\ref{thm:coarsetofine} and Proposition~\ref{prop:HilbIsom}, while the geometric properties of $\Hilb^{[n]}(\mathbb{C}^2/\Gamma)_{\rm red}$ are all inherited from its manifestation as $\mathfrak{M}_{\theta_0}$ via Lemma~\ref{lem:Mthetageometry}.
 
 Next we prove the statement about the resolution. In the notation of \cite[Theorem~1.2]{BC18}, the extremal ray $\rho_1^\perp\cap \dots \cap \rho_r^\perp$ of the cone $F$ that contains $\theta_0 = (-n,1,0,\dots,0)$ lies in the closure of precisely one chamber, namely the chamber $C_+$. Under the isomorphism $L_F$ from \emph{ibid.}, it follows that there is exactly one projective symplectic resolution of $\mathfrak{M}_{\theta_0}$, namely the fine moduli space $\mathfrak{M}_\theta$ for $\theta\in C_+$. By Theorem~\ref{thm:VVW}, this resolution is indeed $\mathfrak{M}_\theta\cong n\Gamma\text{-}\Hilb(\CC^2)$.
 
To see the last statement of the Corollary, consider the morphism $\tau_J\colon \mathfrak{M}_\theta \to\mathcal{M}(A_J)$ constructed in Lemma~\ref{lem:tauJ}. 
This is obtained by restricting a representation of the framed preprojective algebra $\Pi$ to the vertices $0,\infty$, noting that the map of vector bundles $\mathcal{R}_0\to \mathcal{R}_\infty$ is the zero map, and thus we indeed get a representation of~$A_J$. On the other hand, as we discussed before, the isomorphism $\mathfrak{M}_\theta \cong n\Gamma\text{-}\Hilb(\CC^2)$ identifies a $\Pi$-module with the quotient  $\CC[x,y]/I$ for a $\Gamma$-invariant ideal $I$ of $\CC[x,y]$. In this language, the restriction to the $0$ vertex is the $\Gamma$-invariant part of the 
quotient $\CC[x,y]/I$. The statement follows.
    \end{proof}
  
 \begin{remark}
 \begin{enumerate}
     \item Irreducibility of $\Hilb^{[n]}(\mathbb{C}^2/\Gamma)$ was first established by Zheng~\cite{Zheng17} through the study of maximal Cohen--Macaulay modules on Kleinian singularities using a case-by-case analysis following the ADE classification.
     \item Uniqueness of the symplectic resolution of $\Hilb^{[n]}(\mathbb{C}^2/\Gamma)$ was previously known in the special case $n=2$ by the work of Yamagishi~\cite[Proposition~2.10]{Yamagishi17}.
     \item Our approach does not shed light on whether  $\Hilb^{[n]}(\mathbb{C}^2/\Gamma)$ is reduced in its natural scheme structure, coming from its moduli space interpretation. 
     \end{enumerate}
 \end{remark}

\begin{remark}
\label{rem:n=1} For $n=1$, the statement of Theorem~\ref{thm:mainintro} is well known because $\Hilb^{[1]}(\CC^2/\Gamma)\cong \CC^2/\Gamma$, while the statement of Theorem~\ref{thm:coarsetofine} is a framed version of \cite[Theorem~1.2]{CIK18} for $\Gamma\subset \SL(2,\CC)$. Nevertheless, the approach of the current paper is valid for $n=1$ and shows in particular that $\mathfrak{M}_{\theta_J}\cong \CC^2/\Gamma$ for $J=\{0\}$. In fact, this result follows from \cite[Proposition~7.11]{BC18}. Indeed, \emph{ibid}.\ constructs a surjective linear map $L_{C_+}\colon \Theta_v\to N^1(S/(\CC^2/\Gamma))$ with kernel equal to the subspace spanned by $(-1,1,0,\dots,0)$, such that $L_{C_+}(C_+)$ is the ample cone of $S$ over $\CC^2/\Gamma$. Since $\theta_J=(-1,1,0,\dots,0)$ for $J=\{0\}$ and $n=1$, it follows that $\mathfrak{M}_{\theta_J}\cong \CC^2/\Gamma$ in that case. In addition, this explicit description of the kernel of $L_{C_+}$ for $n=1$ shows that the morphisms $\pi_{J,J'}$ and $\tau_{J,J'}$ from Propositions~\ref{prop:poset} and \ref{prop:poset2} are isomorphisms if and only if $J'\setminus J = \{0\}$.
\end{remark}

\appendix

\section{Bounding the dimension vectors of \texorpdfstring{\(\theta_J\)}{θJ}-stable modules}

\label{sec:the ugly proof}

\subsection{The key statement}
 We use the term `diagram' to mean `framed extended Dynkin diagram', and use the notation $ A_i, D_i, E_i $ for the framed extended versions of these Dynkin diagrams. An $A$-module $M$ of the 
 appropriate type naturally determines a representation $V$ of the corresponding quiver $Q^*$ that satisfies the relations from equation \eqref{eqn:Arelations}; we will call these simply 'quiver representations' below. The notion of $\theta_J$-stability for $M$ defines a
 notion of $\theta_J$-stability for $V$. 
		
 For $i\in Q_0^*=\{\infty, 0,1,\dots, r\}$ we write $ v_i \coloneqq\dim_i V $, and for $0\leq i\leq r$ we write $\delta_i\coloneqq \dim(\rho_i)$, so that the regular representation $\delta = \sum_{0\leq i\leq r} \delta_i \rho_i$ coincides with the minimal imaginary root of the 
affine Lie algebra associated to the extended Dynkin diagram. 
 
 The goal of this appendix is to prove the following result, which we require in the proof of Lemma~\ref{lem:existsemistable}. 
 
 \begin{prop}
 \label{prop:dimensionEstimateStableRep}
 Let $J\subseteq \{0,1,\dots, r\}$ be a non-empty subset. Assume that $V$ is a $ \theta_J $-stable quiver representation with $ v_\infty =1 $ and $v_i = n\delta_i$ for $i\in J$ and some fixed natural number $n$. Then $v_j \leq n\delta_j$ for $j\not\in J\cup \{\infty\}$.
 \end{prop}

The proof splits into two cases according to whether or not $0\in J$. We first treat the case $0\in J$ that is required for our conclusions about $\Hilb^{[n]}(\CC^2/\Gamma)$. We then go on to study the case $0\not\in J$ in a  lengthy case-by-case analysis beginning in Section~\ref{sec:remainingCases}.

 Our main tool for proving Proposition~\ref{prop:dimensionEstimateStableRep} is the following estimate, the proof of which is inspired by a result of Crawley-Boevey~\cite[Lemma~7.2]{CrawleyBoevey01}. This inequality is the only consequence of $\theta_J$-stability that we use in the subsequent numerical argument.

 \begin{lem}
 \label{lem:taken from CrawBoe}
Let $ V $ be a $ \theta_J $-stable quiver representation. If $ i\not\in J $, then $2v_i\le \sum_{\{a\in Q_1 \mid \head(a)=i \}} v_{\tail(a)}$.
 \end{lem}
 \begin{proof}
 Define 
 \[
	V_\oplus\coloneqq\bigoplus_{\substack{ a\in Q_1,\\ \head(a)=i }} V_{\tail(a)}.
	\]
	The maps in $V$ determined by arrows with tail and head at vertex $i$ combine to define maps $f\colon V_i\to V_\oplus$ and $g\colon V_\oplus\to V_i$ satisfying $g\circ f = 0$.
	
	If $\ker(f) \neq 0$, then $V$ admits a nonzero subrepresentation $W$ such that $W_i=\ker(f)$ and $W_j=0$ for $j\neq i$. 
	But then $W$ is a proper, nonzero subrepresentation of $V$ satisfying $\theta_J(W)=0$, thereby contradicting $\theta_J$-stability of $V$.
Thus $f$ is injective. Similarly, if $\im(g)\subsetneq V_i$, then $V$ admits a subrepresentation $U$ such that $U_i=\im(g)$ and $U_j=V_j$ for $j\neq i$. Then $U$ is a proper, nonzero subrepresentation of $V$ satisfying $\theta_J(U)=\theta_J(V)=0$ 
	which again contradicts 
	$\theta_J$-stability of $V$, so $g$ is surjective. It follows that the complex
	\begin{equation}\label{eq:stabilityComplex}
	0 \longrightarrow V_i\stackrel{f}{\longrightarrow} V_\oplus \stackrel{g}{\longrightarrow} V_i \longrightarrow 0
	\end{equation}
	has nonzero homology only at $V_\oplus$, so $\dim V_\oplus \geq 2\dim V_i$.
\end{proof}

\begin{proof}[Proof of Proposition~\ref{prop:dimensionEstimateStableRep} in the case $0\in J$]
Let $v'$ be the restriction of the stable dimension vector $v$ to the underlying extended Dynkin diagram, and let $C$ be the Cartan matrix of the same extended diagram. Define $u = v'-n\delta$. We will be done once we show that $u_i\le0$ for all $i$.

We can rephrase Lemma~\ref{lem:taken from CrawBoe} as saying that $(Cv')_i<0\textrm{ for $i \not\in J$}$. Since $C\delta = 0$, this also implies that $(Cu)_i\le 0$ for all $i$. As in Remark~\ref{rem:leftwithDynkin}, removing $J\cup\{\infty\}$ from the diagram leaves a collection of finite-type Dynkin diagrams. Let $Q'$ be any such subdiagram, let $C_{Q'}$ be its Cartan matrix and let $u_{Q'}$ be the restriction of $u$ to $Q'$. As $u_i=0$ for any $i\in J$,
it follows that $(C_{Q'}u_{Q'})_i\le 0$ for all $i$. Now $C_{Q'}^{-1}$ has only positive coefficients (see, e.g. \cite[1.157]{Rosenfeld}), and so 
\[u_{Q'}=C_{Q'}^{-1}C_{Q'}u_{Q'}\]
also satisfies $u_{Q',i}\le 0$ for all $i\in Q'$. Then $u_i\le 0$ for all $i$, giving $v'_i\leq n\delta_i$ for all $i$ as required. 
\end{proof}

\begin{remark}
Our original proof of Proposition~\ref{prop:dimensionEstimateStableRep} in the case $0\in J$ used a lengthy case-by-case argument, similar to that which follows for the case $0\not\in J$. We are grateful to the referee for suggesting this more elegant approach (see also \cite[p.6]{Nakajima20}). 
Unfortunately, we were unable to extend this technique to the case $0\not\in J$. Indeed, if we define $V_\oplus$ as the sum of all vector spaces indexed by adjacent vertices in the McKay quiver - i.e. excluding the framing vertex - then the complex~\eqref{eq:stabilityComplex} can 
have nonzero homology at the second $V_i$. 
\end{remark}

\subsection{Strategy and preparatory results for the case $0\not\in J$}\label{sec:remainingCases}
We now lay the foundation for the proof of Proposition~\ref{prop:dimensionEstimateStableRep} in the case $0\not\in J$. We argue by contradiction, performing a case-by-case analysis on Dynkin diagrams. The basic idea is as follows. First, if the inequality $v_i > n\delta_i$ holds for a vertex $i$ but not its neighbour $j$, we deduce a basic inequality~\eqref{eqn:basic} and show that this inequality can be `pushed along' the branches of the diagram (see Lemma~\ref{lem:basic inequalities}). If the diagram branches at a trivalent vertex, then we push the inequality further along at least one branch (see Lemma~\ref{lem:branching}). This leads either to a contradiction or to strong constraints on $ \dim V $. 

	\begin{lem}
	\label{lem:basic inequalities}
		\begin{enumerate}
			\item[\one] Let $ i, i-1 $ be adjacent vertices of the diagram. If $ v_i>n\delta_i $ and $ v_{i-1}\le n\delta_{i-1} $, then 
			\begin{equation}\label{eqn:basic}
			\delta_{i-1}v_i>\delta_iv_{i-1}.
			\end{equation}
			\item[\two] Suppose the vertex $i\not\in J$ is bivalent, and neither of its neightbours is $\infty$:
			\begin{equation}
			\begin{tikzpicture}[start chain]
			\dydots
			\dnode{i-1}
			\dnode{i}
			\dnode{i+1}
			\dydots
			\end{tikzpicture}
			\end{equation}
			Then $ \delta_{i-1}v_i>\delta_iv_{i-1} $ implies $ \delta_{i}v_{i+1}>\delta_{i+1}v_{i} $. If in addition $ v_i>n\delta_{i} $,  
			then $ v_{i+1}>n\delta_{i+1} $. 
		\end{enumerate}
	\end{lem}
\begin{proof}
Part~\one\ is immediate. Since $i$ and $\infty$ are not neighbours, $ 2\delta_i =\delta_{i-1}+\delta_{i+1}$ holds. Part~\two\ follows by combining this equality with 
the assumed inequality $ \delta_{i-1}v_i>\delta_iv_{i-1} $ and $ 2v_i\le v_{i-1}+v_{i+1} $ coming from Lemma~\ref{lem:taken from CrawBoe}. The last statement is again immediate.
\end{proof}

\begin{lem}\label{lem:branching}
	Suppose that the diagram has a trivalent vertex $i\not\in J$, not adjacent to the vertex $\infty$:
	\begin{equation}
	\begin{tikzpicture}
	\begin{scope}[start chain]
	\dydots
	\dnode{i-1}
	\dnode{i}
	\dnode{j}
	\dydots
	\end{scope}
	\begin{scope}[start chain=br going above]
	\chainin(chain-3);
	\dnodebr{k}
	\updots
	\end{scope}
	\end{tikzpicture}
	\end{equation} 
	and assume that $  \delta_{i-1}v_i>\delta_iv_{i-1}$. 
 \begin{enumerate}
\item[\one] At least one of the inequalities $ \delta_jv_i<\delta_iv_j $ and $\delta_{k} v_i< \delta_{i}v_k$ must hold.
\end{enumerate}
Suppose now that $v_i>n\delta_i$, that  $ \delta_{j}v_i< \delta_{i}v_j$ holds, and furthermore that the branch starting at $j$ does not branch further. Then

\begin{enumerate}[resume]

\item[\two] the branch starting at $ j $ does not contain any vertices in $ J $, and
\item[\three] the same branch must terminate at the framing vertex $\infty$, and in this case $ \delta_{i}v_j=\delta_jv_{i}+1 $.
\end{enumerate}
\end{lem}

\begin{remark}
 The only framed extended Dynkin diagrams where a trivalent vertex is adjacent to the framing vertex are of type $A_i$ for $i>1$. We handle the case of such a vertex not being in $J$ in Lemma A.8.
\end{remark}

\begin{proof}
For \one, combining $2\delta_{i}=\delta_{i-1}+\delta_j+\delta_{k} $ with $2v_i\le v_{i-1}+v_k+v_j $ and $ \delta_{i-1}v_i>\delta_iv_{i-1} $ leads to $\delta_{j}v_i+\delta_{k} v_i<\delta_i v_j+\delta_{i}v_k $ which implies the result. For \two\ and \three, we denote the vertices as 
		\begin{equation}
	\begin{tikzpicture}
	\begin{scope}[start chain]
	\dydots
	\dnode{i}
	\dnode{j}
	\dydots
	\dnode{j+l-1}
	\dnode{j+l}
	\end{scope}
	\begin{scope}[start chain=br going above]
	\chainin(chain-2);
	\updots
	\end{scope}
	\end{tikzpicture}
	\end{equation} 
	if the branch does not contain the framing vertex, or
	\begin{equation}
	\begin{tikzpicture}
	\begin{scope}[start chain]
	\dydots
	\dnode{i}
	\dnode{j}
	\dydots
	\dnode{j+l-1}
	\dnode{j+l}
	\dnode{\infty}
	\end{scope}
	\begin{scope}[start chain=br going above]
	\chainin(chain-2);
	\updots
	\end{scope}
	\end{tikzpicture}
	\end{equation} 
    if it does. To simplify notation, we take $ j-1=i $ in the following argument. 
    One of the following must occur. 
	\begin{itemize}
\item \emph{The branch contains another vertex in $ J $.}
	Suppose that $ j^\prime $ is the node with smallest index on the branch such that $ j^\prime \ne i $ and $ j^\prime\in J $. Lemma~\ref{lem:basic inequalities}\two\ gives $ \delta_{j^\prime-1}v_{j^\prime}>\delta_{j^\prime}v_{j^\prime-1} $ and $ v_{j^\prime}>n\delta_{j^\prime} $, contradicting $j^\prime\in J$.
	
\item \emph{The branch contains no vertices in $ J\cup \infty $.}	Repeated applications of Lemma~\ref{lem:basic inequalities}\two\ show that $ \delta_{j+l-1}v_{j+l}>\delta_{j+l}v_{j+l-1} $. However, since $ 2\delta_{j+l}=\delta_{j+l-1} $, this implies $ 2v_{j+l}>v_{j+l-1} $, contradicting Lemma~\ref{lem:taken from CrawBoe}.
	
\item \emph{The branch contains no vertices of $J$, and terminates at $\infty $.}
   We prove a slightly stronger statement, namely that for any vertex $m\neq \infty $ on the branch, we have $\delta_{m-1}v_{m} =\delta_{m}v_{m-1} +1$. We proceed by induction on the number of edges that lie between $\infty$ and $m$. For the base case $m=j+l$, note that $ \delta_{j+l-1}v_{j+l}>\delta_{j+l}v_{j+l-1} $ implies $ 2v_{j+l}>v_{j+l-1} $. However, since $ 2v_{j+l}\le v_{j+l-1}+1 $ by Lemma~\ref{lem:taken from CrawBoe}, we must have $ 2v_{j+l}= v_{j+l-1}+1 $. If there is more than one edge between $\infty$ and $m$, then the induction hypothesis gives $\delta_{m}v_{m+1}=\delta_{m+1}v_{m}+1 $.
   Combining this with $2v_m\le v_{m-1} + v_{m+1}$ from Lemma~\ref{lem:taken from CrawBoe} and $2\delta_m = \delta_{m+1}+\delta_{m-1}$ shows that $\delta_{m-1}v_m\le \delta_{m}v_{m-1} +1$. Lemma~\ref{lem:basic inequalities}\two\ gives $ \delta_{m-1}v_{m} >\delta_{m}v_{m-1} $ and the result follows.
\end{itemize}
This concludes the proof.
\end{proof}

\subsection{Proof for the case $0\not\in J$, types $A_1$ and $D_4$}
\label{sec:A1D4}
\begin{lem}\label{lem:final2cases}
Proposition~\ref{prop:dimensionEstimateStableRep} holds for $ A_1 $ and $ D_4 $.
\end{lem} 

\begin{proof} For type $A_1$, we have the diagram
\begin{equation}
\begin{tikzpicture}[start chain]
\dnode{\infty}
\dnode{0}
\dnode[ch]{1}
\path (chain-2) -- node[anchor=mid] {{\raisebox{-0.1ex}{$=\joinrel=$}}} (chain-3);
\end{tikzpicture}
\end{equation} where the symbol {\raisebox{-0.1ex}{$=\joinrel=$}}  indicates that the diagram has two edges. The only remaining case is  $J =\{1\}$. A straightforward adaptation of Lemma~\ref{lem:taken from CrawBoe} shows that if $J = \{1\}$, then $2v_0\le 2v_1+1=2n+1$, giving $v_0\le n$.

For type $ D_4 $, the diagram is:

\begin{equation}
\begin{tikzpicture}
\begin{scope}[start chain]
\dnode{\infty}
\dnode{0}
\dnodebr{2}
\dnode{3}
\end{scope}
\begin{scope}[start chain=br going above]
\chainin(chain-3);
\dnodebr{1}
\end{scope}
\begin{scope}[start chain=br going below]
\chainin(chain-3);
\dnode{4}
\end{scope}
\end{tikzpicture}.
\end{equation}
 Suppose without loss of generality that $ 1 \in J$. Then any other vertex $ i $ with $ v_i>n\delta_{i} $ will, by Lemma~\ref{lem:branching} or Lemma~\ref{lem:taken from CrawBoe} give that $ v_2>2n $. The same lemmas show that 
\begin{equation} 4v_2\le 2v_0+2v_1+2v_3+2v_4\le 2n + 3v_2 +1 \label{eq_last_ineq}\end{equation}
and thus $ v_2\le 2n+1 $. So $ v_2 =2n+1 $, but then Lemma~\ref{lem:taken from CrawBoe} gives that $ v_1=v_3=v_4 = n $. Plugging this into~\eqref{eq_last_ineq} gives $ 6n + 3 = 3v_2\le 6n +1 $, a contradiction.
\end{proof}

\subsection{Proof when $0\not\in J$, the general case}
\label{sec:generalcase}
For the rest, we need to handle each diagram type separately.

\begin{lem}\label{lem:dimEstimateTypeA}
	Proposition~\ref{prop:dimensionEstimateStableRep} holds for any diagram of type $ A_i $ with $ i>1 $.
\end{lem}
\begin{proof}
	We number the vertices as follows:
	\begin{equation}
	\begin{tikzpicture}[start chain,node distance=1ex and 2em]
	\dnode{r}
	\dnode{r-1}
	\dydots
	\dnode{2}
	\dnode{1}
	\begin{scope}[start chain=br going above]
	\chainin(chain-3);
	\node[ch,join=with chain-1,join=with chain-5,label={below:\(\scriptstyle{0}\)}] {};
	\dnodeu{\infty};
	\end{scope}

	\end{tikzpicture}.		
	\end{equation}
	Assume that some vertex $ k'\ne \infty $ has $ v_{k'}>n\delta_{k'} =n $. Since $ 0\not\in J $, we may consider a subdiagram
	\begin{equation}
	\begin{tikzpicture}[start chain]
	
	\dydots
	\dnode{i}
	\dydots
	\dnode{r}
	\dnode{0}
	\dnode{1}
	\dydots
	\dnode{j}
	\dydots
	\begin{scope}[start chain=br going above]
	\chainin(chain-5);
	\dnodebr{\infty}
	\end{scope}
	\end{tikzpicture}		
	\end{equation}
	where $ i,j$ (possibly equal) are the only vertices in $ J $, with $ k' $ some vertex in this subdiagram.  We can without loss of generality assume $ 0\le k'<j $. Then there are adjacent vertices $ k, k+1 $ such that $ k'\le k,\  k+1\le j $ with $ v_k>n\ge v_{k+1}$. Repeatedly applying Lemma~\ref{lem:basic inequalities} gives
	\begin{equation}\label{eqdoubleineq}
	v_0>v_1>n.
	\end{equation}  
	There must also be adjacent vertices $ l, l+1 $ between $ i $ and $ 0 $ such that $ v_{l+1}>v_l$. 
In a similar way, this leads to $v_0>v_r$. Combining with~\eqref{eqdoubleineq}, we deduce $2v_0> v_1+v_r +1 $, contradicting Lemma~\ref{lem:taken from CrawBoe}.  
\end{proof}

\begin{lem}\label{lem:dimEstimateTypeD}
	Proposition~\ref{prop:dimensionEstimateStableRep} holds for diagrams of type $ D_i $, $ i>4 $.
\end{lem}

\begin{proof}
	We number the vertices as follows:
	\begin{equation}
	\begin{tikzpicture}
	\begin{scope}[start chain]
	\dnode{\infty}
	\dnode{0}
	\dnodebr{2}
	\dydots
	\dnode{r-2}
	\dnode{r-1}
	\end{scope}
	\begin{scope}[start chain=br going above]
	\chainin(chain-5);
	\dnodebr{r}
	\end{scope}
	\begin{scope}[start chain=br going below]
	\chainin(chain-3);
	\dnode{1}
	\end{scope}
	\end{tikzpicture}
	\end{equation}
	
 Our proof of Proposition~\ref{prop:dimensionEstimateStableRep} for the case $0\in J$ leaves only three possible configurations for the nodes in $J$ when $0\not\in J$, up to symmetry of the diagram. We prove Proposition~\ref{prop:dimensionEstimateStableRep} by contradiction in each case. 
\begin{enumerate}[leftmargin=*,align=left] 
	\item \emph{There is an $i$ such that $ 2\le i\le r-1 $,  $ v_i>n\delta_i = 2n $, and all $ j\in J $ have $ i<j $.} Let $ k $ be  maximal among the vertices such that $ v_k>n\delta_k $. If $ k\le r-2 $, we have $ \delta_{k+1}v_k>\delta_kv_{k+1} $. Otherwise, $ k = r-1 $. We must have $ J = \{r\} $, and by Lemma~\ref{lem:branching} we get $ \delta_{r-3}v_{r-2}>\delta_{r-2}v_{r-3} $.  By symmetry, the case $ k = r$ also leads to $ \delta_{r-3}v_{r-2}>\delta_{r-2}v_{r-3} $.

    Both cases lead (by Lemma~\ref{lem:basic inequalities}) to $ v_2>v_3 $, that is, $ v_2-1 \ge v_3  $. Then Lemma~\ref{lem:branching} gives $ 2v_0 = v_2+1 $ and $ 2v_1\le v_2 $. Combining these with         Lemma~\ref{lem:taken from CrawBoe} leads to 
    \[ 4v_2\le 2v_3 + 2v_1+2v_0\le 4v_2-1, \] which is absurd.
	
	\item \emph{$ v_1>n\delta_1=n $, and all $ j\in J $ have $ j>2 $.} This implies $v_2>2n$. Let $j$ be the least vertex such that $v_j\le n\delta_j$. Applying Lemmas~\ref{lem:basic inequalities}~and~\ref{lem:branching} to the vertices $j-1,j$ (or if $j=r$, the vertices $r, r-2$) we again find $v_2>v_3$. Then the conclusion of case (1) applies.
	
	\item \emph{$ v_0>n\delta_0 $, and all $ j\in J $ have $ j\ge2 $.} If $ 2\in J $, we have $ v_2=2n $, and then $ v_1>n $ leads to $ 2v_1 >2n+1 $, contradicting Lemma~\ref{lem:taken from CrawBoe}. If $ 2\not\in J $ we can again take $j$ as the least vertex with $v_j\le n\delta_j$ and argue as in case (2).
\end{enumerate}

\noindent Therefore, if $D_i$ has $i>4$, then  Proposition~\ref{prop:dimensionEstimateStableRep} holds.
\end{proof}

To conclude, we consider the diagrams $E_6$, $E_7$ and $E_8$. As the proof strategies for these are very similar, we only include the full argument for the $E_8$ case.

\begin{lem}\label{lem:dimEstimateE8}
Proposition~\ref{prop:dimensionEstimateStableRep} holds for type $ E_8 $.
\end{lem}
\begin{proof}

	We number the vertices as follows:
		\begin{equation}
	\begin{tikzpicture}
	\begin{scope}[start chain]
	\dnode{1}
	\dnode{3}
	\dnode{4}
	\dnode{5}
	\dnode{6}
	\dnode{7}
	\dnode{8}
	\dnode{0}
	\dnode{\infty}
	\end{scope}
	\begin{scope}[start chain=br going above]
	\chainin(chain-3);
	\dnodebr{2}
	\end{scope}
	\end{tikzpicture}
	\end{equation}
	
	This time, we split the possible configurations for $J$ across the diagram in the case $0\not\in J$ into four possibilities. Let $ k $ be the minimal vertex with $ v_k>n \delta_k $. The possible configurations are:
	
	\begin{enumerate}[leftmargin=*,align=left]
	\item \emph{$ k>4 $ and all $ j\in J $ have $ j<k $,  or $k =0$.} By Lemma~\ref{lem:basic inequalities} and Lemma~\ref{lem:branching}, we find that $v_0>n\delta_0=n$. The same lemmas show that $\delta_{k+1}v_k     + 1 = \delta_kv_{k+1}$. Let us temporarily use the designation $9$ for the vertex marked $0$. By Lemma~\ref{lem:taken from CrawBoe}, we get
        \[2\delta_kv_k\le \delta_kv_{k+1}+\delta_kv_{k-1}\le \delta_{k+1}v_k+1+\delta_k n\delta_{k-1}\]
        implying $\delta_{k-1}(v_k-n\delta_k )\le1$. But this contradicts $v_k>n\delta_k$.
        
	\item \emph{$ k=4 $ and all $ j\in J $ have $ j<k $}: We have
        \[2v_4\le v_2 +v_3 +v_5\le n\delta_2 +n\delta_3 +v_5 = 7n+v_5.\]
        Since we also have (Lemma~\ref{lem:branching}) $5v_4 +1 =6v_5$, this implies that $7v_5-2\le 35n $. But since $v_5>5n$, this is impossible.
        
	\item \emph{$ k =2 $ and at least one of the vertices $1$ and $ 3 $ are in $J $}: By Lemma~\ref{lem:taken from CrawBoe}, we must have $ v_4\ge6n+1 $. By Lemma~\ref{lem:basic inequalities} applied to the vertex chain $ 1,3,4 $, we find $ 6v_3<4v_4 $. Then Lemma~\ref{lem:branching} shows that $ 6v_5 =5v_4+1 $. Now, if $ v_3\le n\delta_3 $, the same lemma and Lemma~\ref{lem:taken from CrawBoe} imply
        \[12v_4\le 6v_2+6v_3+6v_5\le 8v_4+24n+1 \] leading to $ 24n +4 \le 4v_4\le 24n+1 $, a contradiction. 
        So suppose that $ 1\in J $, and $ v_3> 4n $. By Lemma~\ref{lem:basic inequalities}, we get $ 6v_3<4v_4 $.  As above, we find
        \[ 12v_4\le 6v_2+6v_3+6v_5\le 8v_4 + 6v_3 +1 \]
	    leading to $ 4v_4\le 6v_3 +1 $. This implies that $4v_4 =6v_3+1$, which has no integer solutions. Hence we have a contradiction.
	    
	\item \emph{$ k=1 $ or $ k=3 $, and so $ J $ only consists of $ 2 $}:
		Suppose that $ v_2 =n\delta_2 = 3n $. Then, by Lemma~\ref{lem:branching} and Lemma~\ref{lem:basic inequalities}, we get $ v_4>4n $, say $ v_4 = 4n+t,\ t>0 $. But then Lemma~\ref{lem:branching} and Lemma~\ref{lem:taken from CrawBoe} give
        \[ 12v_4\le 6v_2 + 6v_3+6v_5 \le 18n+4v_4+5v_4+1 \]
        leading to $ 18n+3t=3v_4\le 18n+1 $, a contradiction.
	\end{enumerate}

\noindent Therefore Proposition~\ref{prop:dimensionEstimateStableRep} holds for $E_8$.
\end{proof}

\begin{proof}[Proof of Proposition~\ref{prop:dimensionEstimateStableRep} in the case $0\not\in J$]
Our case-by-case analysis is given in Lemmas~\ref{lem:final2cases}--\ref{lem:dimEstimateE8} where, as noted above, the $E_6$ and $E_7$ cases are similar to that of $E_8$ from Lemma~\ref{lem:dimEstimateE8}.
\end{proof}

\bibliographystyle{alpha}
\bibliography{hilbkleinian}

\begin{thebibliography}{RVdB89}

\bibitem[Aus86]{Auslander86}
M.~Auslander.
\newblock Rational singularities and almost split sequences.
\newblock {\em Trans. Amer. Math. Soc.}, 293(2):511--531, 1986.

\bibitem[BC20]{BC18}
G.~Bellamy and A.~Craw.
\newblock Birational geometry of symplectic quotient singularities.
\newblock {\em Invent. Math.}, 222(2):399--468, 2020.

\bibitem[Bri13]{Brion13}
M.~Brion.
\newblock Invariant {H}ilbert schemes.
\newblock In {\em Handbook of moduli. {V}ol. {I}}, volume~24 of {\em Adv. Lect.
  Math. (ALM)}, pages 64--117. Int. Press, Somerville, MA, 2013.

\bibitem[BS16]{BS16}
G.~Bellamy and T.~Schedler.
\newblock Symplectic resolutions of quiver varieties and character varieties,
  2016.
\newblock To appear \emph{{S}electa {M}ath.}, available at
  \texttt{arXiv:1602.00164}.

\bibitem[Buc12]{Buchweitz12}
R.-O. Buchweitz.
\newblock From {P}latonic solids to preprojective algebras via the {M}c{K}ay
  correspondence, 2012.
\newblock \emph{Oberwolfach Jahresbericht} p18-28, available from
  \texttt{https://publications.mfo.de/handle/mfo/475}.

\bibitem[CB01]{CrawleyBoevey01}
W.~Crawley-Boevey.
\newblock Geometry of the moment map for representations of quivers.
\newblock {\em Compositio Math.}, 126(3):257--293, 2001.

\bibitem[CBH98]{CBH98}
W.~Crawley-Boevey and M.P. Holland.
\newblock Noncommutative deformations of {K}leinian singularities.
\newblock {\em Duke Math. J.}, 92(3):605--635, 1998.

\bibitem[CIK18]{CIK18}
A.~Craw, Y.~Ito, and J.~Karmazyn.
\newblock Multigraded linear series and recollement.
\newblock {\em Math. Z.}, 289(1-2):535--565, 2018.

\bibitem[DH98]{DH98}
I.V. Dolgachev and Y.~Hu.
\newblock Variation of geometric invariant theory quotients.
\newblock {\em Inst. Hautes \'Etudes Sci. Publ. Math.}, (87):5--56, 1998.
\newblock With an appendix by Nicolas Ressayre.

\bibitem[FP04]{FP04}
V.~Franjou and T.~Pirashvili.
\newblock Comparison of abelian categories recollements.
\newblock {\em Doc. Math.}, 9:41--56, 2004.

\bibitem[GNS18]{GNS}
\'{A}. Gyenge, A.~N\'{e}methi, and B.~Szendr\H{o}i.
\newblock Euler characteristics of {H}ilbert schemes of points on simple
  surface singularities.
\newblock {\em Eur. J. Math.}, 4(2):439--524, 2018.

\bibitem[Kin94]{King94}
A.~D. King.
\newblock Moduli of representations of finite-dimensional algebras.
\newblock {\em Quart. J. Math. Oxford Ser. (2)}, 45(180):515--530, 1994.

\bibitem[KN90]{KN90}
P.~Kronheimer and H.~Nakajima.
\newblock Yang-{M}ills instantons on {ALE} gravitational instantons.
\newblock {\em Math. Ann.}, 288(2):263--307, 1990.

\bibitem[Kuz07]{Kuznetsov07}
A.~Kuznetsov.
\newblock Quiver varieties and {H}ilbert schemes.
\newblock {\em Mosc. Math. J.}, 7(4):673--697, 767, 2007.

\bibitem[McK80]{McKay80}
J.~McKay.
\newblock Graphs, singularities, and finite groups.
\newblock In {\em The {S}anta {C}ruz {C}onference on {F}inite {G}roups ({U}niv.
  {C}alifornia, {S}anta {C}ruz, {C}alif., 1979)}, volume~37 of {\em Proc.
  Sympos. Pure Math.}, pages 183--186. Amer. Math. Soc., Providence, R.I.,
  1980.

\bibitem[Nak94]{Nakajima94}
H.~Nakajima.
\newblock Instantons on {ALE} spaces, quiver varieties, and {K}ac-{M}oody
  algebras.
\newblock {\em Duke Math. J.}, 76(2):365--416, 1994.

\bibitem[Nak09]{Nak_branching}
H.~Nakajima.
\newblock Quiver varieties and branching.
\newblock {\em SIGMA Symmetry Integrability Geom. Methods Appl.}, 5:Paper 003,
  37, 2009.

\bibitem[Nak20]{Nakajima20}
H.~Nakajima.
\newblock Euler numbers of {H}ilbert schemes of points on simple surface
  singularities and quantum dimensions of standard modules of quantum affine
  algebras, 2020.
\newblock \texttt{arXiv:math/2001.03834}.

\bibitem[Ros97]{Rosenfeld}
B.~Rosenfeld.
\newblock {\em {Geometry of Lie groups}}.
\newblock Dordrecht: Kluwer Academic Publishers, 1997.

\bibitem[RVdB89]{RVdB89}
I.~Reiten and M.~Van~den Bergh.
\newblock Two-dimensional tame and maximal orders of finite representation
  type.
\newblock {\em Mem. Amer. Math. Soc.}, 80(408):viii+72, 1989.

\bibitem[ST11]{SavTing}
A.~Savage and P.~Tingley.
\newblock Quiver {G}rassmannians, quiver varieties and the preprojective
  algebra.
\newblock {\em Pacific J. Math.}, 251:393--429, 2011.

\bibitem[Tha96]{Thaddeus96}
M.~Thaddeus.
\newblock Geometric invariant theory and flips.
\newblock {\em J. Amer. Math. Soc.}, 9(3):691--723, 1996.

\bibitem[VV99]{VV99}
M.~Varagnolo and E.~Vasserot.
\newblock On the {$K$}-theory of the cyclic quiver variety.
\newblock {\em Internat. Math. Res. Notices}, (18):1005--1028, 1999.

\bibitem[Wan99]{Wang99}
W.~Wang.
\newblock Hilbert schemes, wreath products, and the {M}c{K}ay correspondence,
  1999.
\newblock \texttt{arXiv:math/9912104}.

\bibitem[Wil92]{Wilson92}
P.M.H. Wilson.
\newblock The {K}\"{a}hler cone on {C}alabi-{Y}au threefolds.
\newblock {\em Invent. Math.}, 107(3):561--583, 1992.

\bibitem[Yam17]{Yamagishi17}
R.~Yamagishi.
\newblock Symplectic resolutions of the {H}ilbert squares of {ADE} surface
  singularities, 2017.
\newblock \texttt{arxiv:1709.05886}.

\bibitem[Zhe17]{Zheng17}
X.~Zheng.
\newblock The {H}ilbert schemes of points on surfaces with rational double
  point singularities, 2017.
\newblock \texttt{arxiv:1701.02435}.

\end{thebibliography}

\end{document}